\documentclass[preprint,11pt]{elsarticle}
\usepackage{lipsum}
\makeatletter
\def\ps@pprintTitle{
 \let\@oddhead\@empty
 \let\@evenhead\@empty
 \def\@oddfoot{}
 \let\@evenfoot\@oddfoot}
 \renewcommand{\MaketitleBox}{
  \resetTitleCounters
  \def\baselinestretch{1}
  \begin{center}
    \def\baselinestretch{1}
    \Large \@title \par
    \vskip 18pt
    \normalsize\elsauthors \par
    \vskip 10pt
    \footnotesize \itshape \elsaddress \par
  \end{center}
  \vskip 12pt
}

\usepackage{amssymb}
\usepackage{amsthm}
\usepackage{lineno}
\usepackage{amsfonts}
\usepackage{amsmath}
\usepackage{graphicx}
\usepackage{cases}
\usepackage{epstopdf}
\usepackage{caption}
\usepackage{algorithmic}
\biboptions{sort&compress}

\ifpdf
  \DeclareGraphicsExtensions{.eps,.pdf,.png,.jpg}
\else
  \DeclareGraphicsExtensions{.eps}
\fi

\usepackage{amssymb}
\usepackage{mathtools}

\usepackage[mathscr]{eucal}
\usepackage{color}
\usepackage{geometry}
\usepackage{tikz}
\usetikzlibrary{shapes,arrows}
\usetikzlibrary{decorations.pathreplacing,angles,quotes}

\usepackage{todonotes}

\usepackage{dsfont}
\usepackage{cleveref}
\usepackage{verbatim}
\usepackage{graphicx}
\usepackage{subfig}
\usepackage{tikz}
\usepackage{pgfplots}

\usepgfplotslibrary{groupplots}

\newcommand{\mb}{\mathbb}

\newcommand{\mc}{\mathcal}

\newcommand{\mfT}{{\mathcal{T} }}

\newcommand{\mT}{\mathrm{T}}
\newcommand{\mS}{\mathrm{S}}
\newcommand{\mfF}{{\mathfrak{F}}}

\newcommand{\inlinefrac}[2]{#1/#2}
\newtheorem{theorem}{Theorem}[section]
\newtheorem{lemma}[theorem]{Lemma}
\newtheorem{proposition}[theorem]{Proposition}

\numberwithin{equation}{section}
\newtheorem{definition}{Definition}[section]

\newtheorem{remark}{Remark}
\newtheorem{assumption}{A\!\!}[section]

\begin{document}
\begin{frontmatter}

\title{Hilbert Space-Valued LQ Mean Field Games: An Infinite-Dimensional Analysis\footnote{D. Firoozi would like to acknowledge the support of the Natural Sciences and Engineering Research Council of Canada
(NSERC), grants RGPIN-2022-05337. H. Liu would like to acknowledge the support of GERAD and IVADO through, respectively,  Doctoral Scholarship 2023-2024, and the NSERC-CREATE Program on Machine Learning in Quantitative Finance and Business Analytics (Fin-ML CREATE). }}

\author[1]{Hanchao Liu}
\author[2]{Dena Firoozi}

\address[1]{Department of Decision Sciences, HEC Montréal, Montreal, QC, Canada\\ (email: hanchao.liu@hec.ca)}

\address[2]{Department of Statistical Sciences, University of Toronto, Toronto, ON, Canada\\ (email: dena.firoozi@utoronto.ca)}

\begin{abstract}
This paper presents a comprehensive study of linear-quadratic (LQ) mean
field games (MFGs) in Hilbert spaces, generalizing the classic LQ MFG theory to scenarios involving $N$ agents with dynamics governed by infinite-dimensional stochastic equations. In this framework, both state and control processes of each agent take
values in separable Hilbert spaces. All agents are coupled through the average state of the population which appears in their linear dynamics and quadratic cost functional. Specifically, the dynamics of each agent incorporates
an infinite-dimensional noise, namely a $Q$-Wiener process, and an unbounded operator. The diffusion coefficient of each agent is stochastic involving the state, control, and average state processes. We
first study the well-posedness of a system of $N$ coupled semilinear infinite-dimensional
stochastic evolution equations establishing the foundation of
MFGs in Hilbert spaces. We then specialize to $N$-player LQ games described above and study the asymptotic behaviour as the number of agents, $N$, approaches infinity. We develop an infinite-dimensional variant of the Nash Certainty Equivalence principle and characterize a unique Nash equilibrium for the limiting MFG. Finally, we study the connections between the $N$-player game and the limiting MFG, demonstrating that the resulting limiting best-response strategies form an $\epsilon$-Nash equilibrium for the $N$-player game in Hilbert spaces. This property is established by showing that the mean field approximates the empirical average state within $\epsilon=o(\inlinefrac{1}{\sqrt{N}})$ when agents follow these strategies.
\end{abstract}


\end{frontmatter}

\section{Introduction}\label{sec_intro}

Mean field game (MFG) theory concerns the study and analysis of dynamic games involving a large number of indistinguishable agents who are asymptotically negligible. In such games, each agent is weakly coupled with other agents through the empirical distribution of their states or control inputs. The mathematical limit of this distribution, as the number of agents approaches infinity, is referred to as the mean field distribution. In these games, the behavior of agents in large populations, along with the resulting equilibrium, may be approximated by the solution of corresponding infinite-population games (see, e.g., \cite{huang2006large,huang2007large,lasry2007mean,  carmona2018probabilistic, bensoussan2013mean, cardaliaguet2019master}). 

Originally developed in finite-dimensional spaces, MFGs have become pivotal in addressing large-scale problems involving numerous interacting agents, and have found extensive applications in economics and finance (see, e.g. \cite{carmona2013mean,FirooziISDG2017, Firoozi2022MAFI, gomes_mean-field_2018, chang2022Systemic, casgrain_meanfield_2020,carmona2022carbon}). 

However, there are scenarios where Euclidean spaces do not adequately capture the essence of a problem such as non-Markovian systems. A clear and intuitive example is systems involving time delays. For instance, consider the interbank market model initially introduced in \cite{carmona2013mean}, where the logarithmic monetary reserve (state) of each bank is driven by its rate of borrowing or lending (control action). An extension of this model, studied in \cite{fouque2018mean}, incorporates a scenario where each bank must make a repayment after a specific period, drawing inspirations from \cite{carmona2018systemic}. This modification introduces delayed control actions into the state dynamics. Consequently, the state process is lifted to an infinite-dimensional function space (for Markovian lifting of stochastic delay equations see e.g. \cite{da2014stochastic}). However, due to a gap in the literature on infinite-dimensional MFGs, \cite{fouque2018mean} merely assumes the existence of the mean field in the infinite-population model.  

Beyond practical motivations, investigating MFGs in infinite-dimensional spaces offers an interesting mathematical perspective due to the distinctive treatment required compared to Euclidean spaces. In such spaces, the evolution of a stochastic process is governed by an infinite-dimensional stochastic equation (see e.g. \cite{da2014stochastic, gawarecki2010stochastic}). These equations, also termed stochastic partial differential equations (SPDEs), form a powerful mathematical framework for modeling dynamical systems with infinite-dimensional states and noises. In other words, these equations describe the evolution of random processes in infinite-dimensional Hilbert spaces. The extension to infinite dimensions becomes essential when dealing with phenomena that exhibit spatial or temporal complexities at various scales, such as fluid dynamics and quantum field theory. 

Single-agent optimal control problems in Hilbert spaces have been well-studied in the past (see, e.g.,  \cite{ichikawa1979dynamic,hu2022stochastic,tessitore1992some,nurbekyan2012calculus,gomes2015minimizers} for the LQ setting.). Recent works \cite{dunyak2022linear, dunyak2024quadratic} develop a framework in which optimal control problems over large-size networks are approximated by infinite-dimensional stochastic equations driven by $Q$-noise processes in $L^2[0,1]$. Moreover, McKean-Vlasov control problems in Hilbert spaces have recently been studied in \cite{cosso2023optimal}. To the best of our knowledge, there are only a few works related to mean field games in Hilbert spaces. Besides \cite{fouque2018mean} (where the noise processes are real Brownian motions), the contemporaneous work \cite{federico2024linear} studies a limiting mean field game system involving a constant volatility and a cylindrical Wiener process in Hilbert spaces, where the existence and uniqueness of the solution on both small and arbitrary time intervals are discussed.

The goal of this paper is to present a comprehensive study of LQ MFGs in Hilbert spaces, where the state equation of each agent is modeled by an infinite-dimensional stochastic equation
(for classic LQ MFGs with $\mathbb{R}^n$-valued state and control processes extensively studied in the literature we refer to \cite{bensoussan2016linear, huang2007large,huang2010large,firoozi2020convex,liu2023lqg, firoozi2020epsilon,firoozi2022LQG,huang2020linear,Malhame_Toumi_2024,li2023linear,FIROOZI2022-hybrid,FIROOZI2022-exploratory}.). Specifically, we consider an $N$-player game where the dynamics of agents are modeled by coupled linear stochastic evolution equations in a Hilbert space, with coupling occurring through the empirical average of the states. Each agent aims to minimize a quadratic cost functional, which is also affected by the empirical average of states. The contributions of the current paper can be summarized as follows.
\begin{itemize}
    \item[(i)] We study a general $N$-player LQ game. In particular, the state equation of each agent is influenced by the average state in the drift coefficient and incorporates Hilbert space-valued $Q$-Wiener processes and an unbounded system operator. Additionally, the volatility of each agent involves the state, control, and average state processes, resulting in stochastic volatility and multiplicative noise.
    \item[(ii)] To ensure the well-posedness of the Hilbert space-valued $N$-player game described above, we establish regularity results for a system of $N$ coupled semilinear stochastic evolution equations in Hilbert spaces. The dynamics of the $N$-player game studied falls as a special case within this system.
    \item[(iii)] We study the limiting problem where the number of agents goes to infinity. Establishing a Nash equilibrium for this model involves identifying a unique fixed point within an appropriate function space. To achieve this, we develop an infinite-dimensional variant of the Nash Certainty Equivalence. The model studied in \cite{federico2024linear} can be viewed as a special case of the limiting MFG model addressed in the current paper (see \Cref{lah}). However, due to the different methodologies used, our required conditions and results regarding the existence and uniqueness of a fixed point differ from those established in \cite{federico2024linear}.
    \item[(iv)]  We study the connections between the $N$-player game and the limiting mean field game, demonstrating that the Nash equilibrium strategies obtained in the limiting case form an $\epsilon$-Nash equilibrium for the $N$-player game in Hilbert spaces. This property is established by showing that the mean field approximates the empirical average state within $\epsilon=o(\inlinefrac{1}{\sqrt{N}})$ when agents follow these strategies. 
\end{itemize}
The organization of the paper is as follows. \Cref{prelim} introduces the notations and some preliminaries in infinite-dimensional stochastic calculus to ensure the paper is self-contained and accessible. \Cref{cabs} presents the regularity results for coupled stochastic evolution equations in Hilbert spaces. \Cref{HLMFG} addresses MFGs in Hilbert spaces including optimal control in the limiting case, the fixed-point argument, and both Nash and $\epsilon$-Nash equilibria. Finally, \Cref{conclusion} examines a toy model inspired by the model presented in \cite{fouque2018mean} and discusses potential extensions.
\section{Preliminaries in Infinite-Dimensional Stochastic Calculus}
\label{prelim}
\subsection{Notations and Basic Definitions}
We denote by $(H,\left \langle \cdot \right \rangle_{H})$ and $(V, \left \langle \cdot \right \rangle_{V})$ two separable Hilbert spaces (we drop the letter subscripts in the notation when they are clear from the context). By convention, we use $\left|\cdot \right|_{\cdot}$ to denote the norm in usual normed spaces and $\left\|\cdot \right\|_\cdot$ to denote the operator norm in operator spaces. Moreover, we denote the space of all bounded linear operators from $V$ to $H$ by $\mathcal{L}(V,H)$, which is a Banach space equipped with the norm $\left \| \mT \right \|_{\mathcal{L}(V,H)}=\sup_{\left | x \right |_{V}=1}\left | \mT x \right |_{H}$.  Let $\left \{ e_i \right \}_{i \in \mathbb{N}}$ denote an orthonormal basis of $V$, where $\mb{N}$ denotes the set of natural numbers. The space of Hilbert–Schmidt operators from $V$ to $H$, denoted $\mathcal{L}_2(V,H)$,  is defined as $
\mathcal{L}_2(V,H):=\left \{\mT \in \mathcal{L}(V,H):\sum_{i \in\mathbb{N}}\left | \mT e_i \right |_{H}^{2}< \infty   \right \}$, 
where $\left | \mT e_i \right |_{H} = \sqrt{\left \langle \mT e_i, \mT e_i\right \rangle_{H}}$. 
Note that $\mathcal{L}_2(V,H)$ is a separable Hilbert space equipped with the inner product $\left \langle \mT,\mS  \right \rangle_2 :=\sum_{i \in\mathbb{N}}  \left \langle \mT e_i,\mS e_i  \right \rangle_{H}$ for all $\mT,\mS \in \mathcal{L}_2(V,H).$
This inner product does not depend on the choice of the basis.

Moreover, an operator $\mT \in \mathcal{L}(V,H)$ is called trace class if $\mT$ admits the representation $\mT x=\sum_{i \in\mathbb{N}}b_i \left \langle a_i, x \right \rangle_{V}$, where $\{a_i\}_{i \in\mathbb{N}}$ and $\{b_i\}_{i \in\mathbb{N}}$ are two sequences in $V$ and $H$, respectively, such that $\sum_{i \in\mathbb{N}}\left | a_i \right |_{V}\left | b_i \right |_{H} < \infty$. We denote the set of trace class operators from $V$ to $H$ by $\mathcal{L}_1(V,H)$, which is a separable Banach space equipped with the norm $ \left \|\cdot  \right \|_{\mathcal{L}_1(V,H)}$ defined as 
\begin{equation} \notag
\left \|\mT  \right \|_{\mathcal{L}_1(V,H)}:= \inf\left \{ \sum_{i \in\mathbb{N}}\left | a_i \right |_{V}\left | b_i \right |_{H} : \left \{ a_i \right \}_{i \in \mathbb{N}} \in V , \left \{ b_i \right \}_{i \in \mathbb{N}} \in H\, \text{and}\,\,  \mT x=\sum_{i \in\mathbb{N}}b_i \left \langle a_i, x \right \rangle_{V}, \forall x \in V \right \}.
\end{equation}
Moreover, $\mathcal{L}_1(V)$ denotes the space of operators acting on $V$, which may be equivalently expressed as $\mathcal{L}_1(V, V)$. For an operator $Q \in \mathcal{L}_1(V)$, the trace of $Q$ is defined as 
\begin{equation}
\textrm{tr}(Q)=\sum_{i \in\mathbb{N}}\left<Qe_i,e_i \right>. \notag   
\end{equation}   
The series converges absolutely, i.e., $\sum_{i \in\mathbb{N}} \lvert \langle Qe_i, e_i \rangle \rvert < \infty$. Furthermore, it can be shown that $\left | \textrm{tr}(Q)\right | \leq \left\|Q \right\|_{\mathcal{L}_1(V)}$. For more details on Hilbert–Schmidt and trace class operators, we refer the reader to \cite{peszat2007stochastic}, \cite{da2014stochastic} and \cite{fabbri2017stochastic}. 

We use $\mT^\ast$ to denote an adjoint operator which is the unique operator that satisfies \( \langle \mT x, y \rangle = \langle x, \mT^*y \rangle \) for a linear operator $\mT$ and all vectors \( x \) and \( y \) in the appropriate spaces.

Let $(\mathscr{S},\mathscr{A},\mu)$ be a measure space and $(\mathcal{X}, \left | \cdot \right |_{\mathcal{X}})$ be a Banach space. We denote by $L^{p}(\mathscr{S}; \mc{X}), 1\leq p \leq \infty$, the corresponding Bochner spaces, which generalize the classic $L^p(\mathscr{S}; \mathbb{R})$ spaces. For details on Bochner spaces, we refer to \cite{hytonen2016analysis}. We fix the time interval \(\mfT = [0, T]\) with \(T > 0\). We denote $C(\mfT; H)$ as the set of all continuous mappings $h:\mfT \rightarrow H$, a Banach space equipped with the supremum norm, and $C_s(\mfT; \mathcal{L}(H))$ as the set of all strongly continuous mappings $f:\mfT \rightarrow  \mathcal{L}(H)$.

\begin{definition}[$Q$-Wiener Process \cite{da2014stochastic}]
Let $\big(\Omega, \mfF, \mathbb{P}\big)$ be a fixed complete probability space.  
Additionally,  $Q \in \mathcal{L}_1(V)$ and is a positive operator, i.e. $Q$ is self-adjoint and $\langle Qx,x\rangle \geq 0,\, \forall x \in V$. Then, a $V$-valued stochastic process $W = \{W(t): t\in \mfT\}$ is called a $Q$-Wiener process if   
\begin{itemize}
    \item[(i)] $W(0)=0,\, \mathbb{P}-a.s.$,
    \item[(ii)] $W$ has continuous trajectories,
    \item[(iii)] $W$ has independent increments,
    \item[(iv)] $\forall s, t \in \mfT$ such that $0<s<t$, the increment $W(t)-W(s)$ is normally distributed. More specifically,  $W(t)-W(s) \sim \mathcal{N}(0,(t-s)Q)$.\\
\end{itemize} 
\end{definition}
A $V$-valued $Q$-Wiener process $W$ may be constructed as 
\begin{equation} \label{qwie}
W(t)=\sum_{j \in\mathbb{N}}\sqrt{\lambda _{j}}\beta _{j}(t)e_{j},\quad t \in \mfT,  
\end{equation}
where $\left \{ \beta _j \right \}_{j \in\mathbb{N}}$ is a sequence of mutually independent real-valued Brownian motions
defined on a given filtered probability space. Moreover, $\left\{e_{j} \right\}_{j \in\mathbb{N}}$ is a complete orthonormal basis of $V$ and $\left\{\lambda _{i} \right\}_{i \in\mathbb{N}}$ is a sequence of positive numbers that diagonalize the operator $Q$. Moreover, $\left\{\lambda _{i} \right\}_{i \in\mathbb{N}}$ is a sequence of positive numbers, and $\left\{e_{j} \right\}_{j \in\mathbb{N}}$ is a complete orthonormal basis of $V$ that, together, diagonalize the operator $Q$. In other words, for each $j \in \mathbb{N}$, $\lambda_j$ is an eigenvalue of $Q$ corresponding to the eigenvector $e_j$ such that 
\begin{equation}
Qe_j=\lambda _{j}e_j. \notag
\end{equation}
Note that in this case, we have $\textrm{tr}(Q)=\left\|Q \right\|_{\mathcal{L}_1(V)}$ since $Q$ is a positive operator. Moreover, the series in \eqref{qwie} converges in $L^{2}\left ( \Omega ,C\left (\mfT, V \right ) \right )$ \cite{gawarecki2010stochastic}.

Consider the probability space $(\Omega, \mfF, \mc{F}, \mb{P})$ where the filtration $\mc{F}=\left \{{\mathcal{F}_t} : t \in \mfT\right \}$ satisfies the usual condition. Similarly to the literature (see e.g. \cite{da2014stochastic,gawarecki2010stochastic, fabbri2017stochastic}), we assume that $W$, defined in \eqref{qwie}, is a $Q$-Wiener process with respect to $\mc{F}$, i.e. $W(t)$ is $\mathcal{F}_t$-measurable, and $W(t+h)-W(t)$ is independent of $\mathcal{F}_t$, $\forall h \geq 0$, $\forall t, t+h \in \mfT$. 

We denote by $V_Q=Q^{\frac{1}{2}}V$ the separable Hilbert space endowed with the inner product
\begin{equation} \label{Vq}
\left \langle u,v \right \rangle_{V_{Q}}=\sum_{j \in \mathbb{N}}\frac{1}{\lambda_j}\left \langle u,e_j \right \rangle_V\left \langle v,e_j \right \rangle_V,\quad u,v \in V_Q.
\end{equation}
Note that 
\begin{align}  \label{rell2}
&\mathcal{L}(V,H) \subseteq  \mathcal{L}_2(V_Q,H) \\ \notag &
\left \| \mT \right \|_{\mathcal{L}_2(V_Q,H)}^{2} \leq \textup{tr}(Q)\left \| \mT \right \|_{\mathcal{L}(V,H)}^{2}, \quad \forall \mT \in \mathcal{L}(V,H)   
\end{align}
Below, we introduce certain spaces of stochastic processes defined on a filtered probability space $ \big(\Omega, \mfF,\mc{G}=\left\{ \mathcal{G}_t : t \in \mfT\right\}, \mathbb{P} \big)$ with values in a Banach space $(\mathcal{X}, \left | \cdot \right |_{\mathcal{X}})$. 
\begin{itemize}
    \item  $\mathcal{M}^2(\mfT;\mathcal{X})$ denotes the Banach 
space of all $\mathcal{X}$-valued progressively measurable processes $x(t)$ satisfying 
\begin{equation}\label{M2}
    \left|x \right|_{\mathcal{M}^2(\mfT;\mathcal{X})}:=\left (\mathbb{E}\int_{0}^{T}\left|x(t) \right|^{2}_{\mathcal{X}}dt  \right )^{\frac{1}{2}}< \infty.   
\end{equation}
    \item $\mathcal{H}^2(\mfT;\mathcal{X})$ denotes the Banach 
space of all $\mathcal{X}$-valued progressively measurable processes $x(t)$ satisfying
\begin{equation}\label{H2}
\left|x\right|_{\mathcal{H}^2(\mfT;\mathcal{X})}= \left (\displaystyle \sup_{t \in \mfT} \mathbb{E}\left|x(t) \right|_{\mathcal{X}} ^{2}\right )^{\frac{1}{2}}< \infty.
\end{equation} 
\end{itemize}

Obviously, $\mathcal{H}^2(\mfT;\mathcal{X}) \subseteq \mathcal{M}^2(\mfT;\mathcal{X})$. Furthermore, $\mathcal{B}(\mathcal{X})$ denotes the Borel sigma-algebra on the space $\mathcal{X}$. 

\subsection{Controlled Infinite-Dimensional Linear SDEs}\label{ContInfDimSDEs}
We denote by $H$, $U$, and $V$ three real separable Hilbert spaces. We  then introduce a controlled infinite-dimensional stochastic differential equation (SDE) as 
\begin{align} \label{dyna}
 dx(t)&=(Ax(t)+Bu(t)+m(t))dt+(Dx(t)+Eu(t)+v(t)) dW(t) ,\notag \\
 x(0)&= \xi,  
\end{align}
where $\xi \in L^2(\Omega;H)$. Moreover, $x(t) \in H$ denotes the state and $u(t) \in U$ the control action at time $t$. The control process $u=\{u(t): t\in \mfT\}$ is assumed to be in $\mathcal{M}^2(\mfT;U)$. The $Q$-Wiener process $W$ is defined on a filtered probability space $\left (\Omega, \mfF,\left\{ \mathcal{F}_t\right\}_{t \in \mfT}, \mathbb{P} \right )$ and takes values in $V$. The unbounded linear operator $A$, with domain $\mathcal{D}(A)$, is an infinitesimal generator of a $C_{0}$-semigroup $S(t) \in \mathcal{L}(H),\, t \in \mfT$. Moreover, there exists a constant $M_T$ such that 
\begin{equation}
\left\|S(t) \right\|_{\mathcal{L}(H)}   \leq  M_T,\quad \forall t \in \mfT.\label{S_bound}
\end{equation}
 where $M_T := M_A e^{\alpha T}$, with $M_A \geq 1$ and $\alpha \geq 0$ \cite{goldstein2017semigroups}. The choices of $M_A$ and $\alpha$ are independent of $T$. Furthermore, $B \in \mathcal{L}(U,H)$, $D \in \mathcal{L}(H,\mathcal{L}(V,H))$, $E \in \mathcal{L}(U,\mathcal{L}(V,H))$, the process $m\in C(\mfT;H)$, and the process $v \in L^{\infty}(\mfT;\mathcal{L}(V,H))$. 
We focus on the mild solution of \eqref{dyna}.
\begin{definition}[Mild Solution of an Infinite-Dimensional SDE \cite{da2014stochastic}]
A mild solution of \eqref{dyna} is a process $x \in \mathcal{H}^2(\mfT;H) $ such that $\forall t \in \mfT$, we have
\begin{equation} \label{milds}
 x(t)=S(t)\xi +\int_{0}^{t}S(t-r)(Bu(r)+m(r))dr+\int_{0}^{t}S(t-r)(Dx(r)+Eu(r)+v(r)) dW(r),\quad  \mathbb{P}-a.s. 
\end{equation}    
\end{definition}   
For the results on the existence and uniqueness of a mild solution to  \eqref{dyna}, we refer to  \cite[Section 1.4]{fabbri2017stochastic}.
\begin{remark}(\cite[Chapter 6]{da2014stochastic})\label{alt-mild-sol}
Since the diffusion term in \eqref{dyna} takes the form of multiplicative noise, the mild solution can be equivalently expressed as 
\begin{equation} \label{milds2}
 x(t)=S(t)\xi +\int_{0}^{t}S(t-r)(Bu(r)+m(r))dr+\int_{0}^{t}\sum_{j=1}^{\infty }S(t-r)(D_jx(r)+E_ju(r)+v(r)) d\beta_{j}(r),
\end{equation}       
where the bounded linear operators $D_j \in \mathcal{L}(H)$ and $E_j \in \mathcal{L}(U, H) $ are defined as
\begin{equation} \label{diei}
D_j\,x:=\sqrt{\lambda _j}\,(Dx)\,e_j,\quad E_j\,u:=\sqrt{\lambda _j}\,(Eu)\,e_j,\quad  x \in H,\,\, u \in U.    
\end{equation}
Moreover, the operators $D \in \mathcal{L}(H,\mathcal{L}(V,H))$ and $E \in \mathcal{L}(U,\mathcal{L}(V,H))$ may be expressed as 
\begin{equation} \label{infsum}
(Dx)v=\sum_{j=1}^{\infty }\frac{1}{\sqrt{\lambda _j}}\left<v,e_j \right>D_j\,x,\quad (Eu)v=\sum_{j=1}^{\infty }\frac{1}{\sqrt{\lambda _j}}\left<v,e_j \right>E_j\,u,\quad  v \in V.
\end{equation}
\end{remark}
\begin{remark} \label{integrand}
    In general,  stochastic integrals with respect to a $Q$-Wiener process are constructed for suitable processes which take values in $\mathcal{L}_2(V_Q,H)$, see e.g. \cite{da2014stochastic} and \cite{gawarecki2010stochastic}. However, in this paper, as well as in many works concerning infinite-dimensional control problems, the integrand processes can only be $\mathcal{L}(V,H)$-valued. We refer to \cite{curtain1970ito} and \cite{ichikawa1982stability} for the construction of stochastic integrals for suitable $\mathcal{L}(V,H)$-valued processes.  Such constructions are special cases of those presented in \cite{da2014stochastic} and \cite{gawarecki2010stochastic}.
 
\end{remark}

\section{Coupled Controlled Stochastic Evolution Equations in Hilbert Space} \label{cabs}
In classical mean field games (MFGs), the dynamics of the relevant $N$-player game is modeled as a system of finite-dimensional SDEs, the regularities of which are well-studied in the literature.  However, in this paper, the dynamics of the $N$-player game will be modeled as $N$ coupled infinite-dimensional stochastic equations. To be more specific, the state of each agent satisfies an infinite-dimensional stochastic equation which is involved with the states of all agents. The well-posedness of such a system has not been rigorously established in the literature. Thus, we aim to address it in this section. For this purpose, we first discuss the existence of a sequence of independent $Q$-Wiener processes. Next, we prove the existence and uniqueness of the solution to a system of $N$ coupled infinite-dimensional stochastic equations.

More precisely, in the classic setup of MFGs, the individual idiosyncratic noises form a sequence of independent real-valued Brownian motions. In the current context, however, we require a sequence of independent $Q$-Wiener processes. The following proposition examines the existence of such a sequence.

\begin{proposition} \label{sQ}
Let $\left (\Omega, \mfF,\mathbb{P} \right )$ be a probability space and $Q$ be a positive trace class operator on the separable Hilbert space $V$. Then, there exists a sequence of independent $V$-valued $Q$-Wiener processes $\left\{W_{i} \right\}_{i \in \mathbb{N}}$ defined on the given probability space.
\end{proposition}
\begin{proof}
Let $W, \left \{ \beta _j \right \}_{j \in\mathbb{N}}$ be the processes defined in \eqref{qwie}, and the corresponding natural filtrations be defined as $\mc{F}^W=\{\mc{F}^W_t: t \in \mfT\}$ and $\mc{F}^\beta=\{\mc{F}^\beta_t: t \in \mfT\}$, where $\mathcal{F}_{t}^{W}=\sigma ( W(s),0 \leq s \leq t  )$, $\mathcal{F}_{t}^{\beta}=\sigma ( \beta_j(s),0 \leq s \leq t, j \in\mathbb{N}  )=\sigma ( \bigcup_{j\in \mathbb{N}} \sigma(\beta_j(s),0 \leq s \leq t)   )$. Subsequently, the augmented filtrations $\bar{\mc{F}}^W=\{\bar{\mc{F}}^W_t: t \in \mfT\}$ and $\bar{\mc{F}}^\beta=\{\bar{\mc{F}}^\beta_t: t \in \mfT\}$ consist of $\bar{\mathcal{F}}_{t}^{W}=\sigma(\mathcal{F}_{t}^{W}\cup \mathcal{N})$ and $\bar{\mathcal{F}}_{t}^{\beta}=\sigma(\mathcal{F}_{t}^{\beta}\cup \mathcal{N})$. It is evident that $\bar{\mathcal{F}}_{t}^{W}=\bar{\mathcal{F}}_{t}^{\beta}$. By applying the enumeration of $\mathbb{N}\times \mathbb{N}$ to the sequence of mutually independent Brownian motions $\left \{ \beta _j \right \}_{j \in\mathbb{N}}$, we can obtain infinitely many distinct sequences of Brownian motions $ \{ \beta^{i}_j \}_{j \in \mathbb{N}}=\{\beta _{1}^{i}, \beta _{2}^{i}, \dots, \beta _{j}^{i}, \dots \}$, each sequence indexed by $i$. The real-valued Brownian motions $\beta _{j}^{i}$ are mutually independent for all indices $i,j \in \mb{N}$. Now we construct a sequence of Q-Wiener processes $\left\{W_{i}\right\}_{i\in \mathbb{N}}$, where $W_{i}(t)$ is defined by 
\begin{equation} \label{wiei}
W_{i}(t)=\sum_{j\in \mathbb{N}}\sqrt{\lambda _{j}}\beta ^{i}_{j}(t)e_{j},\quad t \in \mfT.   
\end{equation}
For our purpose, it is enough to show that the augmented filtrations $\{ \bar{\mathcal{F}}^{\beta^i} \}_{i \in \mathbb{N}}$ are independent. Recall that $\mathcal{F}^{\beta^{i}}=\{\mathcal{F}_{t}^{\beta^{i}}: t \in \mfT\}$, where $\mathcal{F}_{t}^{\beta^{i}}=\sigma ( \bigcup_{j\in \mathbb{N}} \sigma(\beta_j^{i}(s),0 \leq s \leq t)   )$ and $\{ \beta^{i}_j \}_{j,i \in \mathbb{N}}$ are independent Brownian motions. Then the independence of $\{ \mathcal{F}_{t}^{\beta^i} \}_{i \in \mathbb{N}}, \forall t \in \mfT$, follows from the standard results in measure theory (see for instance \cite[Proposition 10.1.7]{cohn2013measure}), and hence $ \{ \bar{\mathcal{F}}_{t}^{\beta^i}\}_{i \in \mathbb{N}}, \forall t \in \mfT$ are also independent. 
\end{proof}
It is straightforward to verify that the sequence of processes $\left\{W_{i}\right\}_{i \in \mathbb{N}}$, constructed in Proposition \ref{sQ}, are (mutually independent) $Q$-Wiener processes with respect to $\mathcal{F}$. This measurability arises because each $W_i$ is constructed using a subsequence of real-valued Brownian motions generating the original $Q$-Wiener process $W$ given by \eqref{qwie}. Usually, this ``universal" filtration $\mathcal{F}$ is larger than necessary. Below, we construct a reduced filtration. \vspace{0.1cm}\\
\textbf{Reduced Filtration \(\mathcal{F}^{[N]}=\big\{\mathcal{F}^{[N]}_t: t \in \mfT \big\}\):}\textit{
Consider a set $\mc{N}=\{1,2,…,N\}$ and let \(\left\{W_{i}\right\}_{i \in \mathcal{N}}\) be $N$ independent $Q$-Wiener processes constructed in \Cref{sQ}. A reduced filtration \(\mathcal{F}^{[N]}\) may be constructed under which these processes are independent \(Q\)-Wiener processes. Note that the processes \(\left\{W_{i} \right\}_{i \in \mc{N}}\) are constructed as described in \eqref{wiei} using \(N\) sequences of mutually independent Brownian motions \(\{ \beta^{i}_j\}_{j \in \mathbb{N}, i \in \mathcal{N}}\). These $N$ sequences may be combined to form a new sequence of mutually independent  Brownian motions. We then construct a new \(Q\)-Wiener process \(W^N\) using this resulting sequence as in \eqref{qwie} and define \(\mathcal{F}^{[N]}\) as the normal filtration that makes \(W^N\) a \(Q\)-Wiener process. Clearly, this filtration  only makes the processes \(\left\{W_{i} \right\}_{i \in \mathcal{N}}\) independent \(Q\)-Wiener processes and can be smaller than $\mathcal{F}$.}

We are now ready to introduce a system of coupled infinite-dimensional stochastic equations defined on $\big(\Omega, \mfF,\mc{F}^{[N]}, \mathbb{P} \big)$ describing the temporal evolution of the vector process $\textbf{x}=\{\textbf{x}(t)=(x_1(t), x_2(t), \ldots, x_N(t)):\, t\in \mfT\}$. Note that $\textbf{x}$ is an $H^{N}$-valued stochastic process, where $H^{N}$ denotes the $N$-product space of $H$, equipped with the product norm $ \left|\textbf{x}(t) \right|_{H^{N}}=\left (\sum_{i\in \mc{N}} \left|x_i(t) \right|_{H}^{2}  \right  )^{\frac{1}{2}}$.
Subsequently,  $\mathcal{M}^2(\mfT;H^N)$ and  $\mathcal{H}^2(\mfT;H^N)$ are defined as the 
spaces of all $H^N$-valued progressively measurable processes $\textbf{x}$, respectively, satisfying $ \left|\textbf{x}\right|_{\mathcal{M}^2(\mfT;H^N)}<\infty$ and $\left|\textbf{x}\right|_{\mathcal{H}^2(\mfT;H^N)}< \infty$. 

The differential form of a system of coupled infinite-dimensional stochastic equations can be represented by 
\begin{align} \label{casd}
&dx_i(t)=(Ax_i(t)+F_i(t,\textbf{x}(t),u_i(t)))dt+B_i(t,\textbf{x}(t),u_i(t))dW_{i}(t).  \notag \\
&x_i(0)=\xi_i,
\end{align}
where, as defined in \eqref{dyna}, $A$ is a $C_0$-semigroup generator. Moreover, the control action $u_i, i \in \mathcal{N}$, is a $U$-valued progressively measurable process, and the initial condition $\xi_i, i \in \mathcal{N}$, is $H$-valued and $\mathcal{F}_0^{[N]}$-measurable. Moreover, $\left\{W_{i}\right\}_{i\in \mathcal{N}}$ is a set of mutually independent $Q$-Wiener processes, each constructed as in Proposition \ref{sQ} and applied to the filtration $\mathcal{F}^{[N]}$. Furthermore, the family of maps $F_i: \mfT\times H^{N} \times U \rightarrow H$ and $B_i: \mfT\times H^{N} \times U \rightarrow \mathcal{L}_2(V_Q,H), \forall i \in \mathcal{N}$, are defined for all $i\in \mathcal{N}$. 
\begin{assumption}\label{as0} For each $i \in \mathcal{N}$, the initial condition $\xi_i$ belongs to $L^2(\Omega;H)$ and is  $\mathcal{F}_0^{[N]}$-measurable.   
\end{assumption}
We focus on the solution of \eqref{casd} in a mild sense, which is defined below. 
\begin{definition}(Mild Solution of Coupled Infinite-Dimensional SDEs) A process $\textbf{x} \in \mathcal{H}^{2}(\mfT;H^N)$, where $\textbf{x}=\{\textbf{x}(t)=(x_1(t), x_2(t), \ldots, x_N(t)):\, t\in \mfT\}$, is said to be a mild solution of \eqref{casd} if, for each \(i \in \mathcal{N}\), the process \(x_i\) is defined \(\mathbb{P}\)-almost surely by the integral equation
\begin{equation} \label{mildcoup}
x_i(t)=S(t)\xi _i+\int_{0}^{t}S(t-r)F_i(r,\textbf{x}(r),u_i(r))dr+\int_{0}^{t}S(t-r)B_i(r,\textbf{x}(r),u_i(r))dW_{i}(r), \,\,\,\forall t \in \mfT,    
\end{equation}
where \(S(t)\) is the $C_0$-semigroup generated by $A$.
\end{definition}
We make the following assumptions on the system of coupled stochastic evolution systems described by \eqref{casd} for every $i\in \mathcal{N}$.
\begin{assumption} \label{as1}   
 $u_i \in \mathcal{M}^2(\mfT;U)$.
\end{assumption}
\begin{assumption} \label{as2}
    The mapping $F_i: \mfT\times H^{N} \times U \rightarrow H $ is $\mathcal{B}\left ( \mfT \right ) \otimes \mathcal{B}(H^{N})\otimes \mathcal{B}(U)/ \mathcal{B}(H) $-measurable. 
    \end{assumption}
\begin{assumption}
   The mapping $B_i: \mfT\times H^{N} \times U \rightarrow \mathcal{L}_2(V_Q,H)$ is $\mathcal{B}\left ( \mfT \right )\otimes\mathcal{B}(H^{N}) \otimes \mathcal{B}(U)/\mathcal{B}(\mathcal{L}_2(V_Q,H))$-measurable, where the Hilbert space $V_Q$ is as defined in \Cref{prelim}.
 \end{assumption}  
 \begin{assumption} \label{as5}
 There exists a constant $C$ such that, for every $t \in \mfT$, $u \in U$ and $\textbf{x},\textbf{y} \in H^N$, we have
    \begin{gather}
     \left|F_i(t,\textbf{x},u)-F_i(t,\textbf{y},u) \right|_H+\left\| B_i(t,\textbf{x}, u)-B_i(t,\textbf{y}, u)\right\|_{\mathcal{L}_2}\leq C \left|\textbf{x}-\textbf{y} \right |_{H^{N}}, \notag \allowdisplaybreaks\\
    \left|F_i(t,\textbf{x}, u) \right|^2_H+\left\| B_i(t,\textbf{x}, u)\right\|_{\mathcal{L}_2}^2\leq C^{2}\left (1+\left|\textbf{x} \right|^2_{H^{N}}+ \left|u \right|^{2}_U  \right ).    \notag
    \end{gather}
    \end{assumption}
The following theorem establishes the existence and uniqueness of a mild solution to the coupled stochastic evolution equations given by \eqref{casd}. This result extends Theorem 7.2 in \cite{da2014stochastic}, which addresses the existence and uniqueness of a mild solution for a single stochastic evolution equation without coupling.

\begin{theorem}(Existence and Uniqueness of a Mild Solution) \label{enu}
Under \Cref{as1}-\Cref{as5}, the set of coupled stochastic evolution equations given by \eqref{casd} admits a unique mild solution in the space $\mathcal{H}^{2}(\mfT;H^N)$.     
\end{theorem}
\begin{proof}
The existence and uniqueness of a mild solution can be established through the classic fixed-point argument for a mapping from $\mathcal{H}^{2}(\mfT;H^N)$ to $\mathcal{H}^{2}(\mfT;H^N)$.
To this end, for any given element \( \mathbf{x} \in\mathcal{H}^{2}(\mfT;H^N) \), the operator \( \Gamma \) is defined component-wise as 
\[\Gamma \textbf{x}(t)=\left (\Gamma _1\textbf{x}(t),\,\Gamma _2\textbf{x}(t), \ldots,\,  \Gamma _N\textbf{x}(t) \right ),\]
where each component \( \Gamma_i \mathbf{x}(t) \) is represented by the integral equation
\begin{equation} 
\Gamma _i\textbf{x}(t)=S(t)\xi _i+\int_{0}^{t}S(t-r)F_i(r,\textbf{x}(r),u_i(r))dr+\int_{0}^{t}S(t-r)B_i(r,\textbf{x}(r),u_i(r))dW_{i}(r) . \label{Gamma_i_def}  
\end{equation}
We show that \( \Gamma \) indeed maps \( \mathcal{H}^{2}(\mfT;H^N) \) into itself. The measurability of \eqref{Gamma_i_def} as an $H$-valued process is established using the standard argument found in \cite{da2014stochastic} and \cite{gawarecki2010stochastic}, based on our assumptions. This is because $F_i$ and $B_i$ are progressively measurable processes, valued in $H$ and $\mathcal{L}_2(V_Q,H)$, respectively. We use the inequality $\vert a+b+c \vert^2 \leq 3 \vert a\vert^2 + 3 \vert b\vert^2 + 3 \vert c\vert^2$, for each $i \in \mathcal{N}$ and $ t \in \mfT$, to get
\begin{align} \label{Gamma_i_def1}
\mathbb{E} \left | \Gamma _i\textbf{x}(t) \right |^{2}_H \leq\, & 3\mathbb{E} \left |S(t) \xi _i \right |^2_H+3\mathbb{E}\left [ \left|\int_{0}^{t}S(t-r)F_i(r,\textbf{x}(r),u_i(r))dr \right|^{2}_H \right ]\notag\\ &+3\mathbb{E}\left [ \left|\int_{0}^{t}S(t-r)B_i(r,\textbf{x}(r),u_i(r))dW_{i}(r) \right|^{2}_H \right]. 
\end{align}

For the Bochner integral in \eqref{Gamma_i_def1}, we have 
\begin{align}
\mathbb{E}\left [ \left|\int_{0}^{t}S(t-r)F_i(r,\textbf{x}(r),u_i(r))dr \right|^{2}_H \right ]&\leq T \mathbb{E}\left [\int_{0}^{t}\left|S(t-r)F_i(r,\textbf{x}(r),u_i(r)) \right|^{2}_Hdr \right ]\notag\\
&\leq T \mathbb{E}\left [\int_{0}^{t}\left\Vert S(t-r)\right \Vert_{\mathcal{L}(H)}^2\left|F_i(r,\textbf{x}(r),u_i(r)) \right|^{2}_Hdr \right ]\notag\\
&\leq TM_{T}^{2}C^{2}\mathbb{E}\left[\int_{0}^{t}\left(\left|(\textbf{x}(r) \right|_{H^{N}}^{2}+ \left|u_{i}(r) \right|^{2}_U+1\right)dr\right],
 \label{det_int}   
\end{align}
where the first inequality results from the Cauchy–Schwarz inequality.
For the stochastic integral in \eqref{Gamma_i_def1}, we have 
\begin{align}
\mathbb{E}\left [ \left|\int_{0}^{t}S(t-r)B_i(r,\textbf{x}(r),u_i(r))dW_{i}(r) \right|^{2}_H \right ]&\leq C^{'}\mathbb{E}\left [ \int_{0}^{t}\left\|B_i(r,\textbf{x}(r),u_i(r)) \right\|_{\mathcal{L}_2}^{2}dr   \right ]\notag \\ & \leq C^{2}C^{'}\mathbb{E}\left[\int_{0}^{t}\left(\left|\textbf{x}(r) \right|_{H^{N}}^{2}+\left|u_{i}(r) \right|^{2}_U+1\right)dr\right],
\label{stoch_int} 
\end{align}
where the first inequality is obtained by the standard approximation technique for stochastic convolutions (see e.g. \cite[Theorem 4.36]{da2014stochastic} and \cite[Corollary 3.2]{gawarecki2010stochastic}) in the current context for every $t \in \mfT$ and $i \in \mathcal{N}$.
Note that the constant $C^{'}$ only depends on $T$ and $M_T$ (defined in \eqref{S_bound}). Substituting \eqref{det_int} and \eqref{stoch_int} in \eqref{Gamma_i_def1}, we obtain
\begin{align} \label{boundgam}
\mathbb{E} \left | \Gamma _i\textbf{x}(t) \right |^{2}_H &\leq 3\mathbb{E} \left |S(t) \xi _i \right |^2_H+3(C^{2}C^{'}+TM_{T}^{2}C^{2})\mathbb{E}\left[\int_{0}^{t}\left(\left|\textbf{x}(r) \right|_{H^{N}}^{2}+\left|u_{i}(r) \right|^{2}_U+1\right)dr\right] \notag\\& \leq 3\mathbb{E} \left |S(t) \xi _i \right |^2_H+3(C^{2}C^{'}+TM_{T}^{2}C^{2})\mathbb{E}\left[\int_{0}^{T}\left(\left|\textbf{x}(r) \right|_{H^{N}}^{2}+\left|u_{i}(r) \right|^{2}_U+1\right)dr\right],\ \forall t \in \mfT.
\end{align}
Hence, 
\begin{align} \label{gammabou}
\sum_{i\in \mc{N}}\mathbb{E} \left | \Gamma _i\textbf{x}(t) \right |^{2}_H \leq &\,3 M_{T}^{2} \sum_{i\in \mc{N}} \mathbb{E} \left | \xi _i \right |^2_H+3 (C^{2}C^{'}+TM_{T}^{2}C^{2}) \sum_{i\in \mc{N}}\mathbb{E}\left[\int_{0}^{T}\left(\left|u_{i}(r) \right|^{2}_U\right)dr\right]\notag \\ &+3N(C^{2}C^{'}+TM_{T}^{2}C^{2})\mathbb{E}\left[\int_{0}^{T}\left(\left|\textbf{x}(r) \right|_{H^{N}}^{2}+1\right)dr\right],\quad \forall t \in \mfT.
\end{align}
From \eqref{gammabou}, we have 
\[\left|\Gamma \textbf{x}\right|_{\mathcal{H}^{2}(\mfT;H^N)} = \left (\displaystyle \sup_{t \in \mfT} \sum_{i\in \mc{N}}\mathbb{E}\left|\Gamma _i\textbf{x}(t) \right|_{H}^{2}\right )^{\frac{1}{2}} < \infty.\]
Thus, the transformation \( \Gamma \) is well-defined and maps $\mathcal{H}^{2}(\mfT;H^N)$ into itself. The remaining part of the proof is to show that the mapping $\Gamma$ is a contraction, that is, for any two elements \(\textbf{x} , \textbf{y} \in \mathcal{H}^{2}(\mfT;H^N) \), it holds that
\[\left|\Gamma \textbf{y}- \Gamma \textbf{x}\right|_{\mathcal{H}^{2}(\mfT;H^N)} <\left| \textbf{y}- \textbf{x}\right|_{\mathcal{H}^{2}(\mfT;H^N)}.\]

Using the inequlity $\vert a+b \vert^2 \leq 2 \vert a\vert^2 + 2 \vert b\vert^2$, for each $i \in \mathcal{N}$ and $ t \in \mfT$, we can write 
\begin{align} 
\mathbb{E}\left|\Gamma _i\textbf{y}(t)-\Gamma _i\textbf{x}(t) \right|^{2}_H\leq \, & 2 \mathbb{E}\left [ \left|\int_{0}^{t}S(t-r)(F_i(r,\textbf{x}(r),u_i(r))-F_i(r,\textbf{y}(r),u_i(r)))dr \right|^{2}_H \right ]\notag \\ &+2\mathbb{E}\left [ \left|\int_{0}^{t}S(t-r)(B_i(r,\textbf{x}(r),u_i(r))-B_i(r,\textbf{y}(r),u_i(r)))dW_{i}(r) \right|^{2}_H \right ]. \label{approx}  
\end{align}
For the first term on the RHS of \eqref{approx}, $ \forall t \in \mfT$, we obtain 
\begin{align} \label{difF}
\mathbb{E}\left [ \left|\int_{0}^{t}S(t-r)\left(F_i(r,\textbf{x}(r),u_i(r))- F_i(r,\textbf{y}(r),u_i(r))\right)dr \right|^{2}_H \right ] &\hspace{0cm}\leq C^{2}M_{T}^{2}T\mathbb{E}\left [\int_{0}^{T}\left|\textbf{x}(t)- \textbf{y}(t)\right|^{2}_{H^{N}}  dt\right ]\notag \\
 &\hspace{0cm}\leq C^{2}M_{T}^{2}T^{2} \displaystyle \sup_{t \in \mfT} \mathbb{E} \left|\textbf{x}(t)- \textbf{y}(t)\right|^{2}_{H^{N}}\notag\\ &\hspace{0cm}\leq C^{2}M_{T}^{2}T^{2}\left|\textbf{x}-\textbf{y} \right|^{2}_{\mathcal{H}^{2}(\mfT;H^N)}.   
\end{align}
Similarly, for the second term on the RHS of \eqref{approx} we have 
\begin{align} \label{difB}
&\mathbb{E}\left [ \left|\int_{0}^{t}S(t-r)(B_i(r,\textbf{x}(r),u_i(r))-B_i(r,\textbf{y}(r),u_i(r)))dW_{i}(r) \right|^{2}_H \right ] \notag \allowdisplaybreaks\\ &\leq C^{'}\mathbb{E}\left [ \int_{0}^{T}\left\|B_i(r,\textbf{x}(r),u_i(r))-B_i(r,\textbf{y}(r),u_i(r)) \right\|_{\mathcal{L}_2}^{2}dr   \right ] \notag \allowdisplaybreaks\\ & \leq M_{T}^{2}C^{'}C^{2}T(\sup_{t \in \mfT} \mathbb{E} \left|\textbf{x}(t)- \textbf{y}(t)\right|^{2}_{H^{N}}) \notag\allowdisplaybreaks\\
&\leq M_{T}^{2}C^{'}C^{2}T \left|\textbf{x}-\textbf{y} \right|^{2}_{\mathcal{H}^{2}(\mfT;H^N)},\quad \forall t \in \mfT. 
\end{align}
Based on \eqref{approx} to \eqref{difB}, we obtain 
\begin{equation}
\sum_{i\in \mc{N}}\mathbb{E}\left|\Gamma _i\textbf{y}(t)-\Gamma _i\textbf{x}(t) \right|^{2}_H \leq 2 N M_{T}^{2}C^{2}T(C^{'}+T)\left|\textbf{x}-\textbf{y} \right|^{2}_{\mathcal{H}^{2}(\mfT;H^N)} ,\quad \forall t \in \mfT, \notag 
\end{equation}
and subsequently
\begin{align} \label{contrco}
 \left|\Gamma \textbf{y}- \Gamma \textbf{x} \right|^{2}_{\mathcal{H}^{2}(\mfT;H^N)} \leq 2 N M_{T}^{2}C^{2}T(C^{'}+T)\left|\textbf{x}-\textbf{y} \right|^{2}_{\mathcal{H}^{2}(\mfT;H^N)} . 
\end{align}
By employing the same argument as in Theorem 7.2 of \cite{da2014stochastic}, if $T$ is sufficiently small,  then the mapping $\Gamma$ is a contraction. We apply this reasoning on the intervals $\big[ 0,\widetilde{T} \big],\big[ \widetilde{T},2\widetilde{T} \big ], \cdots, \big [(n-1)\widetilde{T},T\big ]$, where $\widetilde{T}$ satisfies \eqref{contrco} and $n\widetilde{T}=T$.     
\end{proof}
\begin{remark} An alternative approach to prove \Cref{enu} could involve formulating the set of $N$ stochastic evolution equations as a single $H^N$-valued equation. This requires defining appropriate operators between the associated spaces and verifying that a valid $H^N$-valued Wiener process can be constructed using $\{W_{i}\}_{i\in \mathcal{N}}$. Such a reformulation also involves technical considerations. Following this, the existence and uniqueness of the solution may be established by adapting existing results related to single stochastic evolution equations.  
\end{remark}

\section{Hilbert Space-Valued LQ Mean Field Games} \label{HLMFG}
\subsection{$N$-Player Game in Hilbert Space}
We consider a differential game in Hilbert spaces defined on $ \big(\Omega, \mfF, \mathcal{F}^{[N]}, \mathbb{P} \big)$, where $\mathcal{F}^{[N]}$ is constructed in \Cref{cabs}. This game involves $N$ asymptotically negligible (minor) agents, whose dynamics are governed by a system of coupled stochastic evolution equations, each given by the linear form of \eqref{casd}. More precisely, the dynamics of a representative agent indexed by $i$, $i \in \mathcal{N}$, are given by 
\begin{align} \label{dyfin}
 dx_i(t)&=(Ax_i(t)+Bu_i(t)+F_1x^{(N)}(t))dt+(Dx_i(t)+Eu_i(t)+F_2x^{(N)}(t)+\sigma) dW_{i}(t) , \\ \notag 
 x_i(0)&=\xi_i, 
\end{align}
where the agents are coupled through the term $x^{(N)}(t):=\frac{1}{N}(\sum_{i\in \mc{N}}x_{i}(t) )$, which represents the average state of the 
$N$ agents. We assume that all agents are homogeneous and share the same operators. Specifically, $F_1 \in \mathcal{L}(H)$, $F_2 \in \mathcal{L}(H;\mathcal{L}(V;H))$ and $\sigma \in \mathcal{L}(V;H) $ and all other operators are as defined in \eqref{dyna}. In addition to \Cref{as0}, we impose the following assumption for the initial conditions.
\begin{assumption}\label{init-cond-iid}The initial conditions $\{\xi_i\}_{i \in \mc{N}}$ are i.i.d. with $\mb{E}[\xi_i]=\bar{\xi}$.
\end{assumption}
\begin{assumption}(Filtration \& Admissible Control)
The filtration available to agent $i$, $i\in \mc{N}$, is $ \mathcal{F}^{[N]}$. Subsequently, the set of admissible control actions for agent $i$, denoted by $\mc{U}^{[N]}$, is defined as the collection of  $\mc{F}^{[N]}$-adapted control laws $u^i$ that belong to $ \mathcal{M}^2(\mfT;U)$.   
\end{assumption}
Clearly, the system described by \eqref{dyfin} satisfies the assumptions \Cref{as1}-\Cref{as5} and its well-posedness is ensured by Theorem \ref{enu}.

Moreover, agent $i,\, i\in \mc{N}$, aims to minimize the cost functional 
\begin{equation} \label{cosfinN}
J^{[N]}_{i}(u_{i},u_{-i})=\mathbb{E}\int_{0}^{T}\left(\left|M^{\frac{1}{2}}\left (x_{i}(t)-\widehat{F}_1x^{(N)}(t)  \right ) \right|^{2}+\left|u_{i}(t) \right|^{2}\right)dt+\mathbb{E}\left|G^{\frac{1}{2}}\left (x_{i}(T)-\widehat{F}_2x^{(N)}(T)  \right ) \right|^{2},  
\end{equation}
where $M$ and $G$  are positive operators on $H$, and $\widehat{F}_1, \widehat{F}_2 \in \mathcal{L}(H)$. We note that a positive operator can be added to the quadratic term associated with the control process in the cost functional. However, this operator must satisfy specific conditions to ensure that its inverse is also a well-defined positive operator. To avoid further complexity in the notation, we use the identity operator in this work.

In general, solving the $N$-player differential game described in this section becomes  challenging, even for moderate values of $N$ and for finite-dimensional cases. The interactions between agents lead to a large-scale optimization problem, where each agent needs to observe the states of all other interacting agents. To address the dimensionality and the information restriction, following the classical MFG methodology, we investigate the limiting problem as the number of agents $N$ tends to infinity. In this limiting model, the average behavior of the agents, known as the mean field, can be mathematically characterized, simplifying the problem. Specifically, in the limiting case, a generic agent interacts with the mean field, rather than a large number of agents. Furthermore, the mean field happens to coincide with the mean state of the agent. In the subsequent sections, we develop the Nash Certainty Equivalence principle and characterize a Nash equilibrium for the limiting game in Hilbert spaces. We then demonstrate that this equilibrium yields an $\epsilon$-Nash equilibrium for the original $N$-player game.

\subsection{Limiting Game in Hilbert Spaces} \label{lah}
In this section we present the limiting game which reflects the scenario where, in system \eqref{dyfin}-\eqref{cosfinN}, the number of agents $N$ tends to infinity. In this case, the optimization problem faced by a representative agent $i$ is described as follows. 
Specifically, the dynamics of a representative agent, indexed by $i$, is given by \begin{align} \label{dylimi-S0}
 dx_i(t)&=(Ax_i(t)+Bu_i(t)+F_1\bar{x}(t))dt+(Dx_i(t)+Eu_i(t)+F_2\bar{x}(t)+\sigma) dW_{i}(t) , \\  \notag 
 x_i(0)&=\xi_i,  
\end{align}
and the cost functional to be minimized by agent $i$ by 
\begin{equation} \label{cosfinlimi-S0}
{J}^\infty_i(u_i)=\mathbb{E}\int_{0}^{T}\left(\left|M^{\frac{1}{2}}\left (x_{i}(t)-\widehat{F}_{1}\bar{x}(t)  \right ) \right|^{2}+\left| u_{i}(t) \right|^{2}\right)dt+\mathbb{E}\left|G^{\frac{1}{2}}\left (x_{i}(T)-\widehat{F}_{2}\bar{x}(T) \right ) \right|^{2},
\end{equation} 
where $\bar{x}(t)$ represents the coupling term in the limit and is termed the mean field. In this context, on the one hand, a Nash equilibrium for the system consists of the best response strategies of the agents to the mean field $\bar{x}(t)$. On the other hand, in the equilibrium where all agents follow Nash strategies, together they replicate the mean field, i.e. $\frac{1}{N}\sum_{i\in \mc{N}}{x}_{i}(t)\stackrel{\text{q.m.}}{\longrightarrow} \bar{x}(t)$.
We impose the following assumption for the limiting problem.  
\begin{assumption}(Filtration \& Admissible Control)\label{Info-set-Adm-Cntrl}
  The filtration $\mathcal{F}^{i,\infty}$ of agent $i$ satisfies the usual conditions and ensures that $W_i$ is a $Q$-Wiener process and that the initial condition $\xi_i$ is $\mathcal{F}^{i,\infty}_0$-measurable. Subsequently, the set of admissible control actions for agent $i$, denoted by $\mc{U}_i$, is defined as the collection of  $\mc{F}^{i,\infty}$-adapted control laws $u^i$ that belong to $ \mathcal{M}^2(\mfT;U)$.  
\end{assumption}

We first, in \Cref{sincon}, treat the interaction term as an input $g \in C(\mfT ;H)$ and solve the resulting optimal control problem for a representative agent given by the dynamics 
\begin{align} \label{dylimi}
 dx_i(t)&=(Ax_i(t)+Bu_i(t)+F_1g(t))dt+(Dx_i(t)+Eu_i(t)+F_2g(t)+\sigma) dW_{i}(t) , \allowdisplaybreaks\\  \notag 
 x_i(0)&=\xi_i,  
\end{align}
and the cost functional 
\begin{equation} \label{cosfinlimi}
J(u_i)=\mathbb{E}\int_{0}^{T}\left(\left|M^{\frac{1}{2}}\left (x_{i}(t)-\widehat{F}_{1}g(t)  \right ) \right|^{2}+\left| u_{i}(t) \right|^{2}\right)dt+\mathbb{E}\left|G^{\frac{1}{2}}\left (x_{i}(T)-\widehat{F}_{2}g(T)  \right ) \right|^{2}.
\end{equation} 
Then, in \Cref{fixp}, we address the consistency condition described by 
\begin{equation} \label{fixpd}
\mathbb{E}\left[x_{i}^{\circ}(t)\right]=g(t),\quad   \forall i \in \mathcal{N},\, t \in \mfT,
\end{equation} 
where $x_{i}^{\circ}$ represents the optimal state of agent $i$ corresponding to the control problem described by \eqref{dylimi}-\eqref{cosfinlimi}.
Finally, in \Cref{eps-Nash}, we show that $\tfrac{1}{N}\sum_{i\in \mc{N}}{x}_{i}^\circ(t)\stackrel{\text{q.m.}}{\longrightarrow}\mathbb{E}[x_{i}^{\circ}(t)]  $. 

Due to the symmetry among all agents, we drop the index $i$ in \Cref{sincon} and \Cref{fixp}, where we discuss the optimal control problem of individual agents and the relevant fixed-point problem. However, in \Cref{Nash}, where Nash equilibrium is discussed, we use the index $i$ to effectively distinguish between agents by their independent trajectories and respective filtrations. Before addressing the limiting problem in these sections, we will introduce, in the next section, some mappings and their Riesz representations that are essential for the discussions.

\subsubsection{Mappings Associated with Riesz Representations}\label{reizo}
In this section, we introduce multiple mappings and their associated Riesz representations that will be used throughout the remainder of the paper. These mappings are the same as those defined in \cite{ichikawa1979dynamic}. However, since the solution of the limiting problem heavily relies on these mappings, we include more details here. We note that in the special case where the state and control processes do not appear in the diffusion coefficient of each agent, these mappings are not required for the analysis.

Recall that $Q$ is a positive trace class operator on the Hilbert space $V$. For any $\mathcal{R} \in \mathcal{L}(H)$, it can be easily verified that the following expressions are bounded bilinear functionals on their corresponding product spaces:
\begin{align*}
&\mathrm{tr}((Eu)^* \mathcal{R}(Dx) Q),\quad \forall x\in H, u \in U, \\
&\mathrm{tr}((Dx)^* \mathcal{R}(Dy) Q), \quad \forall x, y \in H, \\
&\mathrm{tr}((Eu)^* \mathcal{R}(Ev) Q), \quad \forall u, v \in U.   
\end{align*}
Moreover, for any $ \mathcal{P} \in \mathcal{L}(H;V)$, the expressions below are bounded linear functionals on $H$ and $U$, respectively:
\begin{align*}
&\mathrm{tr}(\mathcal{P}\left ( D x\right )Q), \quad \forall x \in H, \allowdisplaybreaks\\
&\mathrm{tr}(\mathcal{P}\left ( Eu \right )Q),\quad \forall u \in U.
\end{align*}
\begin{definition}(Riesz Mappings)\label{Riesz-Rep}
Using the Riesz representation theorem the mappings $\Delta_1: \mathcal{L}(H) \rightarrow \mathcal{L}(H; U) $, $\Delta_2: \mathcal{L}(H) \rightarrow \mathcal{L}(H)$ and $\Delta_3: \mathcal{L}(H) \rightarrow \mathcal{L}(U)$ are defined such that 
\begin{align*}
\mathrm{tr}((Eu)^* \mathcal{R}(Dx) Q)&=\langle \Delta_1(\mathcal{R})x, u \rangle, \quad \forall x \in H, \, \forall u \in U,\quad \Delta_1(\mathcal{R}) \in \mathcal{L}(H; U),\\
\mathrm{tr}((Dx)^* \mathcal{R}(Dy) Q)&=\langle \Delta_2(\mathcal{R})x, y \rangle, \quad \forall x, y \in H, \qquad \quad \,\,\,\, \Delta_2(\mathcal{R}) \in \mathcal{L}(H),\\
\mathrm{tr}((Eu)^* \mathcal{R}(Ev) Q)&=\langle \Delta_3(\mathcal{R})u, v \rangle, \quad \forall u, v \in U, \qquad \quad \,\,\,\,\, \Delta_3(\mathcal{R}) \in \mathcal{L}(U).
\end{align*}
Similarly, the mappings $\Gamma_{1}:\mathcal{L}(H;V) \rightarrow H$ and $\Gamma_{2}: \mathcal{L}(H;V) \rightarrow U$ are defined such that 
\begin{align*}
\mathrm{tr}(\mathcal{P}\left ( D x\right )Q)&=\left<\Gamma_1(\mathcal{P}), x \right>, \quad \forall x \in H,\quad \Gamma_1(\mathcal{P}) \in H,\allowdisplaybreaks\\
\mathrm{tr}(\mathcal{P}\left ( Eu \right )Q)&= \left<\Gamma_2(\mathcal{P}), u \right>, \quad \forall u \in U, \quad \Gamma_2(\mathcal{P}) \in U.
\end{align*}
\end{definition}
In the following proposition, we establish the linearity and boundness of the introduced Riesz mappings. 
\begin{theorem} \label{reizlb}
The mappings $\Delta_k, k=1,2,3$, and $\Gamma_k, k=1,2$, are linear and bounded. Specifically, we have
\begin{align}
 \Gamma_1 \in \mathcal{L}(\mathcal{L}(H;V) ; H) \quad &\text{and}  \quad \left\|\Gamma _1\right\|\leq R_1\quad \text{with}\quad R_1=\mathrm{tr}\left (Q  \right )\left\|D \right\|, \allowdisplaybreaks\\
 \Gamma_2 \in \mathcal{L}(\mathcal{L}(H;V) ;  U) \quad &\text{and}\quad \left\|\Gamma _2\right\|\leq R_2 \quad \text{with}\quad R_2=\mathrm{tr}\left (Q  \right )\left\|E \right\|,\allowdisplaybreaks\\
 \Delta_1 \in \mathcal{L}({\mathcal{L}(H); \mathcal{L}(H;U)}) \quad &\text{and}\quad \left\|\Delta _1\right\|\leq R_3 \quad \text{with}\quad  R_3=\mathrm{tr}\left (Q  \right )\left\|D \right\|\left\|E \right\|,\label{R3}\\
 \Delta_2 \in \mathcal{L}({\mathcal{L}(H); \mathcal{L}(H)}) \quad &\text{and}\quad \left\|\Delta _2\right\|\leq R_{4} \quad \text{with}\quad R_{4}=\mathrm{tr}\left (Q  \right )\left\|D \right\|^2,\\
 \Delta_3 \in \mathcal{L}({\mathcal{L}(H); \mathcal{L}(U)}) \quad &\text{and}\quad \left\|\Delta _3\right\|\leq R_{5} \quad \text{with}\quad R_{5}=\mathrm{tr}\left (Q  \right )\left\|E \right\|^2.
 \end{align}
 Moreover, for any positive operator $\mathcal{R} \in \mathcal{L}(H)$ we have
 \begin{equation} \label{R6}
 \left \| (I+\Delta_3(\mathcal{R}))^{-1}\left (B^{\ast}\mathcal{R}+\Delta_1(\mathcal{R})  \right ) \right \| \leq R_6 \left \| \mathcal{R} \right \|, \quad \text{with} \quad R_{6}=\left\|B \right\|+R_{3}. 
 \end{equation}
\end{theorem}
\begin{proof}
We present the demonstration only for the Riesz mapping $\Delta_1$ and the demonstrations for other Riesz mappings follow by a similar argument. To verify the linear property, consider $\mathcal{R}_{1}, \mathcal{R}_{2} \in \mathcal{L}(H)$ and $a, b \in \mathbb{R}$. For all $x \in H$ and $u \in U$, it is straightforward to check that
\begin{equation}
\mathrm{tr}\big((Eu)^* (a\mathcal{R}_{1}+b\mathcal{R}_{2})(Dx) Q\big)=a\,\mathrm{tr}\big((Eu)^* (\mathcal{R}_1)(Dx) Q\big)+b\,\mathrm{tr}\big((Eu)^* (\mathcal{R}_2)(Dx) Q\big).
\end{equation}
Thus, for all $x \in H$ and $u \in U$, we have
\begin{equation}
\langle \Delta_1(a\mathcal{R}_{1}+b\mathcal{R}_{2})x, u \rangle=\langle (a \Delta_1(\mathcal{R}_{1})+b \Delta_1(\mathcal{R}_{2}))x, u \rangle,
\end{equation}
from which we conclude that 
$\Delta_1(a\mathcal{R}_{1}+b\mathcal{R}_{2})=a \Delta_1(\mathcal{R}_{1})+b \Delta_1(\mathcal{R}_{2})$. Next, by simple calculations, for all $x 
\in H$, for all $u \in U$, and $\mathcal{R} \in \mathcal{L}(H)$, we have 
\begin{equation}
\left|\mathrm{tr}((Eu)^* \mathcal{R}(Dx) Q) \right| \leq \left\| (Eu)^* \mathcal{R}(Dx)\right\|_{\mathcal{L}(V,H)}\left\|Q \right\|_{\mathcal{L}_1(V)} \leq R_3 \left \| \mathcal{R} \right \| \left|x \right|_H\left| u\right|_U. 
\end{equation} 
Thus, by the Riesz representation theorem, we have
\begin{equation}
\left \|\Delta_1 (\mathcal{R})  \right \|=\sup_{\left | x \right |_H=1,\left | u \right |_U=1} \left |\mathrm{tr}\big((Eu)^* \mathcal{R}(Dx) Q\big) \right | \leq R_3 \left \| \mathcal{R} \right \|,
\end{equation}
which implies that $\left\|\Delta _1\right\|\leq R_3$. For the second part, we can easily verify that if $\mathcal{R}$ is a positive operator on $H$, then $\Delta_3(\mathcal{R})$ is also a positive operator on $U$. Consequently, it follows that $\left \| (I+\Delta_3(\mathcal{R}))^{-1} (t)\right \|\leq 1, \forall t \in \mfT$. Thus, we have
\begin{equation}
  \left \| (I+\Delta_3(\mathcal{R}))^{-1}\left (B^{\ast}\mathcal{R}+\Delta_1(\mathcal{R})  \right ) \right \| \leq  \left \| B^{\ast}\mathcal{R}+\Delta_1(\mathcal{R})  \right \|\leq R_6  \left \| \mathcal{R} \right \|.
\end{equation}
\end{proof}  
\subsubsection{Optimal Control of Individual Agents}\label{sincon}
In this section, we address the optimal control problem for a representative agent described by \eqref{dylimi}-\eqref{cosfinlimi}.  Infinite-dimensional LQ optimal control problems have been studied in works such as \cite{ichikawa1979dynamic, tessitore1992some, hu2022stochastic}.
We address our specific problem by presenting the results in a compact and self-contained manner, relying on the existing literature. Due to the symmetry among all agents, we drop the index $i$.

\begin{theorem}[Optimal Control Law]  Consider the mappings $\Delta_k, k=1,2,3,$ and $\Gamma_k, k=1,2,$ given in \Cref{Riesz-Rep}, and suppose \Cref{init-cond-iid} holds. Then, the optimal control law $u^{\circ}$ for the Hilbert-space valued system described by \eqref{dylimi}-\eqref{cosfinlimi} is given by 
\begin{equation} \label{opticon}
u^{\circ}(t) = -K^{-1}(T-t)\left[L(T-t)x(t) + \Gamma_2\big(p^*(t)\Pi(T-t)\big) +B^*q(T-t) \right],
\end{equation}
where 
\begin{equation}\label{Coef-thm}
K(t) = I + \Delta_3(\Pi(t)), \quad L(t) = B^{\ast}\Pi(t) + \Delta_1(\Pi(t)),\quad p(t)=F_2g(t)+\sigma,
\end{equation}
with $\Pi \in C_s(\mfT; \mathcal{L}(H))$, such that $\Pi(t)$ is a positive operator $\forall t \in \mfT $, and $q \in C(\mfT ;H)$, each satisfying, respectively, the operator differential Riccati equation and the linear evolution equation, given by 
\begin{align}  
&\frac{d}{dt}\left<\Pi(t)x, x\right> = 2\left<\Pi(t)x, Ax\right> - \left<L^{\ast}(t)K^{-1}(t)L(t)x, x\right> + \left<\Delta_2(\Pi(t))x, x\right> + \left<Mx, x\right>, \notag \\
& \qquad \quad \,\,\, \Pi(0) = G, \quad x \in \mathcal{D}(A), \label{rica} \\\hspace{-2cm}
&\dot{q}(t) = \left(A^* - L^*(t)K^{-1}(t)B^*\right)q(t) + \Gamma_1\big(p^*(T-t) \Pi(t)\big) - L^*(t)K^{-1}(t)\Gamma_2\big(p^*(T-t) \Pi(t)\big)\notag \\ &\,\,\,\,\qquad + (\Pi(t)F_1 - M\widehat{F}_{1})g(T-t),  
\quad q(0) = -G\widehat{F}_{2}g(T).\label{qt}
\end{align}
\end{theorem}
\begin{proof}
Similar to finite-dimensional LQ control problems, 
the optimal control law involves a Riccati equation but in the operator form, and an offset equation which is an $H$-valued deterministic evolution equation. 
Consider the operator differential Riccati equation given by \eqref{rica}. We refer to  \cite{hu2022stochastic} for the existence and uniqueness of the solution $\Pi(t)$ to \eqref{rica}, as our problem falls within the framework studied in \cite{hu2022stochastic}. Specifically, from \eqref{milds2}, the mild solution of \eqref{dylimi} is in the same form as that of \cite[eq. (2.1)]{hu2022stochastic}. Note that the deterministic terms in the model \eqref{dylimi}-\eqref{cosfinlimi}, i.e. $F_1g(t)$, $F_2g(t)+\sigma$, $\widehat{F}_{1}g(t)$ and $\widehat{F}_{2}g(T)$, do not affect the Riccati equation \eqref{rica}. Moreover, it can be easily verified that $D_j$ and $E_j$ introduced in \Cref{alt-mild-sol} satisfy the conditions specified by \cite[eq. (2.3)]{hu2022stochastic} and \cite[eq. (2.5)]{hu2022stochastic}, respectively. The solution \( \Pi(t) \) is a positive operator on \( H \) for each \( t \in \mfT \), and is strongly continuous on \( \mfT \). Moreover, it is uniformly bounded over the interval \( \mfT \), such that \( \|\Pi(t)\|_{\mathcal{L}(H)} \leq C \) for all \( t \in \mfT \). For the case where $E=0$ in \eqref{dylimi}, we refer to works such as \cite{ichikawa1979dynamic}, \cite{tessitore1992some} and \cite{da1984direct}.
Next, consider the (deterministic) linear evolution equation given by \eqref{qt}.
Given that $\Pi(t)$ and $F_2g(t)+\sigma$ are bounded on the interval $\mfT$, the terms $\Gamma_1((F_2g(T-t)+\sigma)^* \Pi(t))$ and $\Gamma_2((F_2g(T-t)+\sigma)^* \Pi(t))$ are also bounded over $\mfT$. Consequently, the existence and uniqueness of a mild solution to \eqref{qt} follow from the established results for linear evolution equations \cite{diagana2018semilinear}. Specifically, this solution lies within the space $C(\mfT ;H)$.

Now, we begin to solve the corresponding control problem, described by \eqref{dylimi}–\eqref{cosfinlimi}. The mild solution of \eqref{dyna} is expressed as
\begin{equation}  \label{singlemild}
 x(t)=S(t)\xi +\int_{0}^{t}S(t-r)(Bu(r)+F_1g(r))dr+\int_{0}^{t}S(t-r)(Dx(r)+Eu(r)+F_2g(r)+\sigma) dW(r). 
\end{equation}
We introduce a standard approximating sequence given by
\begin{align} \label{yosida}
 dx_n(t)&=(Ax_n(t)+J_n(Bu(t)+F_1g(t)))dt+J_n(Dx_n(t)+Eu(t)+F_2g(t)+\sigma)dW(t) , \notag \\
 x_n(0)&=J_n \xi,
\end{align}
where $J_n=nR(n,A)$, with $R(n,A)=(A-nI)^{-1}$ being the resolvent operator of $A$, is the Yosida approximation of $A$. For more details on Yosida approximation we refer to \cite{da2014stochastic,fabbri2017stochastic}.

Then, the following two standard results hold for \eqref{singlemild} and \eqref{yosida} (see e.g. \cite{ichikawa1982stability, ichikawa1979dynamic}). Firstly, the approximating SDE \eqref{yosida} admits a strong solution, represented as
\begin{equation} \label{sso}
x_n(t)=J_n \xi+\int_{0}^{t}\big(Ax_n(r)+J_n(Bu(r)+F_1 g(r))\big)dr+\int_{0}^{t}J_n(Dx_n(r)+Eu(r)+F_2g(r)+\sigma)dW(r).
\end{equation} 
This means that for each $n$, there exists an adapted process $x_n$ that satisfies the integral form of the approximating SDE almost surely for all $t$ in the interval $\mfT$. Secondly, the sequence of solutions $\{ x_n \}_{n \in \mathbb{N}}$ converges to the mild solution $x$ of the original SDE in the mean square sense uniformly over the interval $\mfT$, i.e.
\begin{equation} \label{convsm}
\lim_{n \to \infty }  \sup_{0\leq t\leq T} \mathbb{E}\left|x_{n}(t)-x(t) \right|^{2}=0   . 
\end{equation}
 
The rest of the proof follows the standard methodology as in \cite{ichikawa1979dynamic}, summarized in the following two steps:
\begin{itemize}
    \item[(i)] Apply Itô's lemma (\cite[Theorem 2.1]{ichikawa1979dynamic}) to $(\langle \Pi(T-t)x_n(t), x_n(t) \rangle + 2 \langle q(T-t), x_n(t) \rangle)$, integrate from $0$ to $T$, and substitute the corresponding terms using \eqref{rica}-\eqref{qt}, \eqref{sso} and \Cref{Riesz-Rep}. Then take the expectation of both sides of the resulting equation.
    \item[(ii)] Take the limit as $n\rightarrow \infty$ of both sides of the expression derived in step (i) and use the convergence property \eqref{convsm}.
\end{itemize}
Note that, compared to finite-dimensional LQ control problems (see e.g. \cite[Section 6.6]{yong1999stochastic}), we must additionally implement step (ii). This is necessary because, in general, Itô's lemma applies only to the strong solutions of infinite-dimensional stochastic equations. 

Finally, by some standard algebraic manipulations, we obtain 
\begin{align*} 
 J(u) &=\mathbb{E}\left<\Pi(T)\xi,\xi \right>+2\mathbb{E}\left<q(T),\xi \right>   + 2\left<G\widehat{F}_{2}g(T), \widehat{F}_{2} g(T) \right>+ \mathbb{E}\Bigg[\int_{0}^{T}\Big|K^{\frac{1}{2}}(T-t)\big[u(t)\notag\\ & \,\,\,+ K^{-1}(T-t)L(T-t)x(t) \notag + K^{-1}(T-t)\left(B^*q(T-t) + \Gamma_2((F_2g(t)+\sigma)^*\Pi(T-t))\right)\big]\Big|^{2}dt\Bigg]\notag \\ &\,\,\,+ \int_{0}^{T} \Bigg[\mathrm{tr}\left((F_2g(t)+\sigma)^*\Pi(T-t)(F_2g(t)+\sigma) Q\right) + \left<M\widehat{F}_{1}g(t),\widehat{F}_{1}g(t) \right> + 2\left\langle q(T-t),F_1g(t) \right\rangle \notag \\\quad &\,\,\,-\left|K^{-\frac{1}{2}}(T-t)\left(B^*q(T-t) + \Gamma_2((F_2g(t)+\sigma)^*\Pi(T-t))\right)\right|^{2}\Bigg]dt. 
\end{align*}
Note that the above equation holds for any initial condition $\xi$ in  $L^2(\Omega;H)$. Therefore, the optimal control law is given by \eqref{opticon}-\eqref{Coef-thm}.
\end{proof}
\subsubsection{Fixed-Point Problem} \label{fixp}
In this section, we address the fixed-point problem described in Section \ref{lah}. From \eqref{dylimi} and under the optimal control given by \eqref{opticon}, the optimal state satisfies
\begin{align} \label{opstate}
x^{\circ}(t) &= S(t)\xi - \int_{0}^{t} S(t-r) \big(BK^{-1}(T-r)L(T-r)x^{\circ}(r) + BK^{-1}(T-r)B^{\ast}q(T-r) + \psi(r)\big) \, dr \notag \\
&\quad + \int_{0}^{t} S(t-r) \big((D - EK^{-1}(T-r)L(T-r))x^{\circ}(r) - EK^{-1}(T-r)B^{\ast}q(T-r) + \phi(r)\big) \, dW(r),
\end{align}

where
\begin{align} 
&\psi(t) = BK^{-1}(T-t)\Gamma_2((F_2g(t)+\sigma)^{\ast} \Pi(T-t))-F_1g(t),\notag \allowdisplaybreaks\\ 
&\phi(t)= -EK^{-1}(T-t)\Gamma_2((F_2g(t)+\sigma)^{\ast} \Pi(T-t))+F_2g(t)+\sigma.
\end{align}
By taking the expectation of both sides in \eqref{opstate}, we obtain the linear evolution equation 
\begin{equation} \label{nt}
\mb{E}\left [x^{\circ}(t) \right ]= S(t)\xi- \int_{0}^{t}S(t-r)(BK^{-1}(T-r)L (T-r)\mb{E}\left [x^{\circ}(r) \right ]+BK^{-1}(T-t) B^{\ast}q(T-t)+\psi (r))dr,
\end{equation}
which admits a unique mild solution in $C(\mfT ;H)$ for any input $g\in C(\mfT ;H)$, using the same argument as that used for \eqref{qt}. Therefore,
\eqref{nt} defines the mapping 
\begin{equation}
\Upsilon: g \in C(\mfT ;H) \longrightarrow \mb{E}[x^\circ(.)] \in C(\mfT ;H).  \label{fixed-point-mapping}
\end{equation}
We show that the mapping $\Upsilon$ admits a unique fixed point. For this purpose, we first establish bounds on $\mfT$ for relevant operators and processes that appear in the mapping $\Upsilon$ characterized by \eqref{nt}.
We start with the operator $\Pi(t)$ which satisfies \eqref{rica}. The following lemma establishes a uniform bound for $\Pi(t)$ across $\mfT$. 
\begin{proposition}[Bound of $\Pi$]\label{Riccati-Bound}
Let $\Pi \in C_s(\mfT; \mathcal{L}(H))$ be the unique solution of the operator differential Riccati equation \eqref{rica}, then we have
\begin{gather}
\left \| \Pi(t) \right \|_{\mathcal{L}(H)} \leq C_1,\quad \forall t \in \mfT,\\ 
C_1:=2M_{T}^{2}\mathrm{exp}(8TM_{T}^{2}\left\|D \right\|^2\mathrm{tr}(Q)) \left ( \left \| G \right \|+T\left \| M \right \|  \right ). \label{eq:pi-bound}
\end{gather}
\end{proposition}
\begin{proof}
For the purpose of illustration, without loss of generality, we introduce a simpler model for which the optimal control law involves the same operator Riccati differential equation as \eqref{rica}. For this specialized model, the dynamics are given by 
\begin{align} 
 dy(t) &= (Ay(t) + Bu(t))dt + (Dy(t) + Eu(t))dW(t) \notag, \\
 y(0)&=\theta \in \mathcal{D}(A),\label{simple_dy}
\end{align}
and the cost functional by
\begin{equation} \label{cosho}
\mathbb{E}\int_{0}^{T}\left(\left|M^{\frac{1}{2}}y(t)\right|^{2} + \left|u(t)\right|^{2}\right)dt + \mathbb{E}\left|G^{\frac{1}{2}}y(T)\right|^{2} ,\notag
\end{equation}
where all the operators are as defined in \eqref{dylimi}-\eqref{cosfinlimi}. The strong solution of the corresponding approximating sequence is given by 
\begin{equation}
y_n(t)=\theta +\int_{0}^{t}\big(Ay_n(r)+J_n(Bu(r))\big)dr+\int_{0}^{t}J_n(Dy_n(r)+Eu(r)) dW(r).    \notag
\end{equation}

Applying Itô's lemma to $\left<\Pi(t-r)y_n(r),y_n(r)\right>$, for $t\in \mfT$, integrating with respect to $r$ from $0$ to $t$, taking the expectation of both sides of the resulting equation and then taking the limit as $n \rightarrow \infty$, we obtain for any admissible control $u$ 
\begin{align}
\left \langle\Pi(t)\theta ,\theta   \right \rangle&=
\mathbb{E}\int_{0}^{t}\left(\left|M^{\frac{1}{2}}y(r)   \right|^{2}+\left|u(r)\right|^{2}\right)dr+\mathbb{E}\left|G^{\frac{1}{2}}y(t)   \right|^{2}-\mathbb{E}\int_{0}^{t}\left|u(r)+K^{-1}L(T-r)y(r) \right|^{2}dr \notag \allowdisplaybreaks\\ & \leq \mathbb{E}\int_{0}^{t}\left(\left|M^{\frac{1}{2}}y(r)   \right|^{2}+\left|u(r)\right|^{2}\right)dr+\mathbb{E}\left|G^{\frac{1}{2}}y(t)   \right|^{2}. \notag 
\end{align}
Setting $u(t) = 0,\, \forall t \in \mfT$, we have 
\begin{equation}
\left \langle\Pi(t)\theta ,\theta   \right \rangle\leq \mathbb{E}\int_{0}^{t} \left|M^{\frac{1}{2}}y_{0}(r)   \right|^{2}dr+\mathbb{E}\left|G^{\frac{1}{2}}y_{0}(t)   \right|^{2} ,  \label{ineq1}
\end{equation}
where $y_{0}(t)$ is the mild solution to $\eqref{simple_dy}$ under $u(t)=0$, satisfying
\begin{align} 
 y_{0}(t)=S(t)\theta +\int_{0}^{t}S(t-r)Dy_{0}(r) dW(r). \notag 
\end{align}
By performing similar computations as in \eqref{stoch_int}, we have 
\begin{align}
\mathbb{E}\left |y_{0}(t)  \right |^{2}
&\leq 2M_{T}^{2}\left |\theta \right|^{2}  +2\mathbb{E}\left |\int_{0}^{t}S(t-r)Dy_{0}(r)dW(r) \right |^{2}\leq   2M_{T}^{2}\left |\theta \right|^{2} +8M_{T}^{2}\mathbb{E}\int_{0}^{t}\left\|Dy_{0}(r) \right\|_{\mathcal{L}_{2}}^{2}dr\notag \allowdisplaybreaks\\& \leq 2M_{T}^{2}\left |\theta \right|^{2} +8M_{T}^{2}\left\|D \right\|^2\mathrm{tr}(Q)\mathbb{E}\int_{0}^{t}\left|y_{0}(r) \right|^{2}dr.\notag
\end{align}
Then, applying Grönwall's inequality, for every $t \in \mfT$, we have 
\begin{equation}
\mathbb{E}\left |y_{0}(t)  \right |^{2} \leq 2M_{T}^{2}\left |\theta \right|^{2} \mathrm{exp}\left(16TM_{T}^{2}\left\|D \right\|^2\mathrm{tr}(Q)\right). \label{ineq2}
\end{equation}
Finally, from \eqref{ineq1} and \eqref{ineq2}, for every $x \in H$, we obtain 
\begin{equation}
\left \langle\Pi(t)\theta,\theta  \right \rangle \leq  \mathbb{E}\int_{0}^{t} \left|M^{\frac{1}{2}}y_{0}(r)   \right|^{2}dr+\mathbb{E}\left|G^{\frac{1}{2}}y_{0}(t)   \right|^{2}    \leq  \left |\theta \right|^{2} C_1,\quad \forall \xi \in H, \notag 
\end{equation}
where $C_1$ is given by \eqref{eq:pi-bound}.  Then, the conclusion follows from the spectral property of self-adjoint operators and the fact that $\mathcal{D}(A)$ is dense in $H$.
\end{proof}
Furthermore, the Riesz mappings $\Delta_k, k = 1,2,3$, and $\Gamma_k, k=1,2$, given in \Cref{Riesz-Rep} and associated with $\Pi(t)$, appear in \eqref{nt}. We can easily apply the results of \Cref{reizlb} and \Cref{Riccati-Bound} to establish the bounds for these operators. This in turn facilitates the determination of bounds for the operators $K^{-1}(t)$ and $K^{-1}(t)L(t)$, both of which are present in \eqref{nt} and defined in \eqref{Coef-thm}. For instance, for every $t \in \mfT$, we have $\left\|\Delta _1( \Pi(t))\right\|\leq C_1 R_3$, where $R_3$ is given by \eqref{R3}, and hence 
\begin{align}
\left\|K^{-1}(t)L(t) \right\| \leq \left\|K^{-1}(t)\right\|\left\|L(t) \right\| \leq \left\|L(t) \right\| \leq C_1R_{6},  
\end{align}
where $R_{6}$ is given by \eqref{R6}.

Now, we establish that the variations of the solution $q \in C(\mfT;H)$ to the the linear evolution equation given by \eqref{qt} are bounded with respect to the variations in the input $g \in C(\mfT ;H)$.
\begin{lemma}[Bounded Variations of $q(t)$ wrt Variations of Input $g(t)$]\label{lemma:bounded-variation}Consider the processes $\Pi \in C_s(\mfT; \mathcal{L}(H))$ and $q \in C(\mfT ;H)$, respectively, satisfying \eqref{rica} and \eqref{qt}. Moreover, let $g_1, g_2 \in C(\mfT ;H)$ be two processes on $\mfT$. Then, we have 
\begin{gather}
\left |q_1-q_2 \right |_{C(\mfT ;H)}\leq \left|g_1-g_2 \right|_{C(\mfT ;H)} M_T(TC_2+\left\|G \right\|\big \|\widehat{F}_2\big \|)e^{M_{T}TC_3}, \label{qbound}\\
C_2:=C_1(R_1\left\|F_2 \right\|+C_1 R_{6}R_{2}\left\|F_2 \right\|+\left\|F_1 \right\|)+\left \| M \right \|\big \|\widehat{F}_1\big \|\label{C2}\\
C_{3}:= C_1 R_{6}\big \| B \big \|,\label{C3}
\end{gather}
where $q_1$ and $q_2$ are the corresponding solutions of \eqref{qt} to the inputs $g=g_1 \in C(\mfT ;H)$ and $g=g_2 \in C(\mfT ;H)$, respectively. 
\end{lemma}
\begin{proof}
The mild solutions $q_i, i =1,2,$ of \eqref{qt} subject to the inputs $g = g_i, i=1,2,$ are given by
\begin{align} \label{mildq}
q_i\left ( t \right )=&-S^{\ast }(t)G\widehat{F}_{2}g_i(T)+\int_{0}^{t}S^{\ast }(t-r)(-L^{\ast }(r)K^{-1}(r)B^{\ast}q_i(r)-M\widehat{F}_{1}g_i(T-r)+\eta_i (r))dr,\notag
\end{align}
where 
\begin{equation}
\eta_i (t)=\Gamma _1((F_2g_i(T-t)+\sigma)^{\ast} \Pi(t))-L^{\ast }(t)K^{-1}(t)\Gamma _2((F_2g_i(T-t)+\sigma)^{\ast} \Pi(t))+\Pi (t)F_1g_i(T-t).     \notag 
\end{equation}
We can show that, $\forall t \in \mfT$, 
\begin{align}
\left | \eta_1 (t)-\eta_2 (t) \right | \leq & \left |\Gamma _1((F_2(g_1(T-t)-g_2(T-t)))^{\ast} \Pi(t))  \right |+\left |\Pi (t)F_1(g_1(T-t)-g_2(T-t))  \right |\notag \\ 
&\quad+\left \|L^{\ast }(t)K^{-1}(t)  \right \|\left |\Gamma _2((F_2(g_1(T-t)-g_2(T-t)))^{\ast} \Pi(t))  \right |\notag \\ 
 \leq & C_1(R_1\left\|F_2 \right\|+C_1 R_{6}R_{2}\left\|F_2 \right\|+\left\|F_1 \right\|)\left|g_1-g_2 \right|_{C(\mfT;H)}.
\end{align}
Thus, $\forall t \in \mfT$, we have,
\begin{align}
\left |q_1(t)-q_2(t)  \right | \leq 
&\left|S^{\ast }(t)G\widehat{F}_{2}(g_1(T)-g_2(T)) \right|+\left|\int_{0}^{t}S^{\ast }(t-r)L^{\ast }(r)K^{-1}(r)B^{\ast}(q_1(r) - q_2(r))dr\right|\notag \\ 
&+\left|\int_{0}^{t}S^{\ast }(t-r)(\eta_1(r) - \eta_2(r))dr\right|+\left|\int_{0}^{t}S^{\ast }(t-r)M\widehat{F}_{1}(g_1(T-r) - g_2(T-r))dr\right| \notag \\
 \leq & M_T(TC_2+\left\|G \right\|\left \|\widehat{F}_2\right \|)\left|g_1-g_2 \right|_{C(\mfT ;H)}+ M_T C_3 \int_{0}^{t} \left|q_1(r)-q_2(r) \right| dr, \label{inter-eq}
\end{align}
where $C_2$ and $C_3$ are, respectively, given by \eqref{C2} and \eqref{C3}. Finally, by applying Grönwall's inequality to \eqref{inter-eq}, we obtain \eqref{qbound}.
\end{proof}
So far, we have demonstrated that all the operators and deterministic processes appearing in the mapping $\Upsilon$, characterized by \eqref{nt}-\eqref{fixed-point-mapping}, are bounded. We may now establish the condition under which this mapping admits a unique fixed point.
\begin{theorem}[Contraction Condition]
The mapping $\Upsilon: g \in C(\mfT ;H) \longrightarrow \mb{E}[x^\circ(.)] \in C(\mfT ;H),$ described by \eqref{nt}, admits a unique fixed point if 

\begin{equation} 
C_4 e^{TM_{T}\left\|B \right\|C_1R_{6}}< 1, \label{concondi}\\
\end{equation}
where
\begin{equation} 
C_4:=T M_T\left (M_{T}\left\|B \right\|^{2}\left(TC_2+\left\|G \right\|\left \|\widehat{F}_2\right \|\right)e^{M_{T}TC_3}+C_1R_2\left\|B \right\|\left\|F_2 \right\|+\left\|F_1 \right\|\right) .\label{cc}
\end{equation}
\end{theorem}
\begin{proof}
Subject to the inputs $g_1, g_2 \in C(\mfT ;H) $, the optimal control characterized in \eqref{opticon} is given by
\begin{equation}
u^{\circ,i}(t)=-K^{-1}(T-t)[L(T-t)x^{\circ,i}(t)+ B^{\ast}q_i(T-t)+\Gamma _2\left((F_2g_i(t)+\sigma)^{\ast} \Pi (T-t) \right)],   \notag
\end{equation}

Subsequently, the expectation of the resulting optimal state $\mathbb{E}[x^{\circ,i}(t)],\, i=1,2,$ satisfies 
\begin{align}
\mathbb{E}[x^{\circ,i}(t)]=S(t)\xi- \int_{0}^{t}&S(t-r)(BK^{-1}(T-r)L (T-r)\mb{E}\left [x^{\circ,i}(r) \right ]\notag\\&+BK^{-1}(T-r) B^{\ast}q_i(T-r)+\psi_i (r))dr, \label{exp-xi}
\end{align}
where 
\begin{equation} \label{tau-i}
\psi_i(t) = BK^{-1}(T-t)\Gamma_2((F_2g_i(t)+\sigma)^{\ast} \Pi(T-t))-F_1g_i(t) 
\end{equation}
From \eqref{tau-i}, $\forall t \in \mfT$, we have
\begin{align}
\left|\psi _1(t)-\psi _2(t)\right|  &\leq \left\|B \right\|\left\|K^{-1}(T-t) \right\|\left|\Gamma _2(F_2(g_1(t)-g_2(t))^{\ast} \Pi (T-t)) \right|+\left\|F_1 \right\|\left|g_1(t)-g_2(t) \right|  \notag \\
& \leq  (C_1R_2\left\|B \right\|\left\|F_2 \right\|+\left\|F_1 \right\|)\left|g_1(t)-g_2(t) \right|.\label{inter-step-2}
\end{align}
Hence,
\begin{align} \label{ditau}
\left| \int_{0}^{t}S(t-r)\left (\psi_1(r)-\psi_2(r) \right )dr  \right|&\leq M_{T}(C_1R_2\left\|B \right\|\left\|F_2 \right\|+\left\|F_1 \right\|) \int_{0}^{t}\left|\psi _1(r)-\psi_2(r)\right|dr\notag \allowdisplaybreaks\\
& \leq T M_T(C_1R_2\left\|B \right\|\left\|F_2 \right\|+\left\|F_1 \right\|)\left|g_1-g_2 \right|_{C(\mfT ;H)}.
\end{align}
By applying the result of \Cref{lemma:bounded-variation}, $\forall t \in \mfT$, we obtain 
\begin{align} \label{diq}
&\left| \int_{0}^{t}S(t-r)BK^{-1}(T-t) B^{\ast}\left (q_1(r)-q_2(r) \right )dr  \right|  \leq T M_T \left\|B \right\|^2 \left |q_1-q_2 \right |_{C(\mfT ;H)}\notag \\ &\qquad\qquad\qquad\leq  TM^2_T\left\|B \right\|^2(TC_2+\left\|G \right\|\left \|\widehat{F}_2\right \|)e^{M_{T}TC_3}\left|g_1-g_2 \right|_{C(\mfT ;H)}. 
\end{align}
Moreover, $\forall t \in \mfT$, we have
\begin{align} \label{diex}
&\left| \int_{0}^{t}S(t-r)BK^{-1}(T-r)L (T-r)\left(\mathbb{E}[x^{\circ,1}(r)]-\mathbb{E}[x^{\circ,2}(r)]\right) dr  \right| \notag\\
&\qquad\leq M_{T}\left\|B \right\|C_1R_{6}\int_{0}^{t}\left|\mathbb{E}[x^{\circ,1}(r)]-\mathbb{E}[x^{\circ,2}(r)] \right|dr, 
\end{align}
From  \eqref{ditau}-\eqref{diex}, $\forall t \in \mfT$, we obtain 
\begin{equation}
\left|\mathbb{E}[x^{\circ,1}(t)]-\mathbb{E}[x^{\circ,2}(t)] \right|\leq C_4 \left|g_1-g_2 \right|_{C(\mfT ;H)}+M_{T}\left\|B \right\|C_1R_{6} \int_{0}^{t}\left|\mathbb{E}[x^{\circ,1}(r)]-\mathbb{E}[x^{\circ,2}(r)] \right|dr.\notag
\end{equation}
Finally, we apply Grönwall's inequality to the above inequality to get
\begin{equation}
\left|\mathbb{E}[x^{\circ,1}(.)]-\mathbb{E}[x^{\circ,2}(.)]\right|_{C(\mfT ;H)} \leq C_4 e^{TM_{T}\left\|B \right\|C_1R_{6}} \left|g_1- g_2\right|_{C(\mfT ;H)},
\end{equation}
from which the fixed-point condition \eqref{concondi} follows.
\end{proof}
We now discuss the feasibility of the contraction condition \eqref{cc}. For this purpose, we do not impose additional assumptions on the operators involved  in \eqref{dylimi} and \eqref{cosfinlimi}, and nor on the $C_0$-semigroup $S(t), t \in \mfT$. 
\begin{proposition}[Contraction Condition Feasibility] \label{convT}
There exists $T>0$ such that the contraction condition \eqref{concondi} holds. 
\end{proposition}
\begin{proof}
From \eqref{S_bound}, the $C_{0}$-semigroup $S(t) \in \mathcal{L}(H),\, t \in \mfT$, is uniformly bounded by a constant $M_T=M_A e^{\alpha T}$.  This constant depends only on $T$, given fixed values of $M_A$ and $\alpha$. Hence, we can treat $M_T$, along with $C_i,i=1,2,3,$ as real-valued functions of $T$. It is evident that $M_T \downarrow M_A$ as $T \downarrow 0$. In addition, we can easily verify that, as $T \downarrow 0$, each $C_i,i=1,2,3,$ monotonically decreases to a positive constant and that $C_{4} \downarrow 0$. Hence, $C_4 e^{TM_{T}\left\|B \right\|C_1R_{6}} \downarrow 0$ as $T \downarrow 0 $. Then from the continuity of the real valued function $C_4 e^{TM_{T}\left\|B \right\|C_1R_{6}}$ with respect to $T$, we conclude that there exists $T > 0$ such that the contraction condition \eqref{cc} holds.
\end{proof}
\begin{remark}[Contraction Condition Feasibility for Fixed $T$]
 \cref{convT} states that for a sufficiently small $T$ the contraction condition \eqref{concondi} holds. This result is consistent with the findings in the finite-dimensional case (see e.g. \cite{bensoussan2016linear,huang2007large}). Moreover, for any fixed $T > 0$ the condition \eqref{concondi} may be satisfied if, for example, $F_1$, $F_2$, $\widehat{F}_1$, and $\widehat{F}_2$ are sufficiently small.
\end{remark}

\subsubsection{Nash Equilibrium}\label{Nash}
The following theorem concludes this section. 
\begin{theorem}[Nash Equilibrium]\label{Nash-eq}Consider the Hilbert space-valued limiting system, described by \eqref{dylimi-S0}-\eqref{cosfinlimi-S0} for $i\in\mathbb{N}$, and the relevant Riesz mappings $\Delta_k, k=1,2,3,$ $\Gamma_k, k=1,2,$ given in \Cref{Riesz-Rep}. Suppose \Cref{init-cond-iid}-\Cref{Info-set-Adm-Cntrl}, and condition \eqref{concondi} hold. Then, the set of control laws $\{u_{i}^{\circ}\}_{i \in \mb{N}}$, where $u_{i}^{\circ}$ is given by
\begin{gather} 
u^{\circ}_i(t)=-K^{-1}(T-t)\left[L(T-t)x_i(t)+\Gamma _2\big((F_2\bar{x}(t)+\sigma)^{\ast}\Pi(T-t)\big)+B^{\ast }q(T-t)\right],\label{nast}\\
K(t) = I + \Delta_3(\Pi(t)), \quad L(t) = B^{\ast}\Pi(t) + \Delta_1(\Pi(t)),\label{nast-2}
\end{gather}
  forms a unique Nash equilibrium for the limiting system where the mean field $\bar{x}(t) \in H$, the operator $\Pi(t) \in \mc{L}(H)$ and the offset term $q(t)\in H$, are characterized by the unique fixed point of the following set of consistency equations 
\begin{align}  
 &\bar{x}(t)=S(t)\bar{\xi}- \int_{0}^{t}S(t-r)\bigg(BK^{-1}(T-r)\Big(L(T-r)\bar{x}(r)+B^{\ast }q(T-r)\notag\\
 &\qquad \,\,\,\,+\Gamma _2\big((F_2\bar{x}(r)+\sigma)^{\ast}\Pi (T-r)\big)\Big)-F_1\bar{x}(r)\bigg)dr, \label{MF-eq}\\  
&\frac{d}{dt}\left<\Pi(t)x, x\right> = 2\left<\Pi(t)x, Ax\right> - \left<L^{\ast}K^{-1}L(t)x, x\right> + \left<\Delta_2(\Pi(t))x, x\right> + \left<Mx, x\right>, \label{Riccati-eq}\qquad \\
&\dot{q}(t) = \left(A^* - L^*(t)K^{-1}(t)B^*\right)q(t) + \Gamma_1\big((F_2\bar{x}(T-t)+\sigma)^{\ast} \Pi(t)\big)\notag \\&\qquad\,\,\, - L^*(t)K^{-1}(t)\Gamma_2\big((F_2\bar{x}(T-t)+\sigma)^{\ast} \Pi(t)\big) +  \left (\Pi(t)F_1- M\widehat{F}_1\right )\bar{x}(T-t),\label{offset-eq}
\end{align}
with $\Pi(0) = G$, $x \in \mathcal{D}(A)$, and $q(0) = -G\widehat{F}_2\bar{x}(T)$. 
\end{theorem}  
\begin{proof}
According to the demonstrations in \Cref{fixp}, if the contraction condition \eqref{concondi} holds, then there exist $\Pi \in C_s\left (\mfT, \mathcal{L}(H)  \right ),\, q \in C(\mfT ;H)$ and $ \bar{x} \in  C(\mfT ;H)$,
which are the unique solution to the set of consistency equations give by \eqref{MF-eq}-\eqref{offset-eq}. In addition, the set of feedback control laws $\{ u_{i}^{\circ}\}_{i \in \mb{N}}$, where $u_{i}^{\circ}$ is given by \eqref{nast}-\eqref{nast-2}, forms a unique Nash equilibrium for the limiting system described by \eqref{dylimi-S0}-\eqref{cosfinlimi-S0}, i.e. 
\begin{equation}\label{Nash-equil}
J_{i}^{\infty}(u_i^{\circ},u_{-i}^{\circ}) = \inf_{u_i \in \mc{U}_i}   J_{i}^{\infty}(u_{i},u_{-i}^{\circ}),\quad  \forall i\in \mathbb{N}.
\end{equation}
This is because, in the limit when the number of agents $N$ goes to infinity, the agents get decoupled from each other and hence the high-dimensional optimization problem faced by agent $i$, $i\in \mathbb{N}$, turns into a single-agent optimal control problem for which there is a unique solution. Hence, agent $i$ cannot improve its cost by deviating from the optimal strategy \eqref{nast}-\eqref{nast-2} and the set of these strategies yields a Nash equilibrium for the limiting system. In other words, the Nash equilibrium property, as defined in \eqref{Nash-equil}, holds trivially for the set of strategies $\{u_{i}^\circ\}_{i \in \mathbb{N}}$ because the limiting cost functional of agent $i$, given by \eqref{cosfinlimi-S0}, is independent of the strategies of other agents.

Subsequently, the equilibrium state of agent $i$ is given by  
\begin{align} \label{opstate-equib}
x^{\circ}_i(t) &= S(t)\xi_i- \int_{0}^{t}S(t-r)\left(BK^{-1}(T-r)L (T-r)x^{\circ}_i(r)+B\tau(r)-F_1\bar{x}(r)\right)dr \notag \allowdisplaybreaks\\
&\quad + \int_{0}^{t}S(t-r)\Big[(D-EK^{-1}(T-r)L (T-r))x^{\circ}_i(r)-E\tau(r)+F_2\bar{x}(r)+\sigma \Big]dW_{i}(r),
\end{align}
where $\tau (t)=K^{-1}(T-t)\left [ B^{\ast }q(T-t)+\Gamma _2((F_2\bar{x}(t)+\sigma)^{\ast}\Pi (T-t))\right]$. Moreover, we have $\mathbb{E}\left [ {x}_{i}^{\circ}(t) \right ]=\bar{x}(t),\,  \forall i \in \mc{N},\,  \forall t \in \mfT$, where $\bar{x}(t)$ is given by \eqref{MF-eq}. We note that \eqref{MF-eq} represents the mild solution to the mean field equation 
\begin{align}
d\bar{x}(t)=\Big[A\bar{x}(t)-BK^{-1}(T-t)\Big(L(T-t)\bar{x}(t)+B^{\ast }q(T-t)+\Gamma_2\big((F_2\bar{x}(t)+\sigma)^{\ast}\Pi (T-t)\big)\Big)\notag\\+F_1\bar{x}(t)\Big]dt.    
\end{align}

\end{proof} 
\subsection{$\epsilon$-Nash Equilibrium}\label{eps-Nash}
In this section, we establish the $\epsilon$-Nash property of the set of strategies given by \eqref{nast}-\eqref{nast-2} for the $N$-player game described by \eqref{dyfin}-\eqref{cosfinN}. Due to the symmetric properties (exchangeability) of agents, we study the case where agent $i=1$ deviates from this set of strategies. Specifically, we suppose that any agent $i$, $i\in\mc{N}$ and $i\neq 1$, employs the feedback strategy $u_i^{[N],\circ}$ given by 
\begin{equation} \label{equil-strategy-N-player}
u^{[N],\circ}_i(t) = -K^{-1}(T - t)\left[L(T - t)\, x^{[N]}_i(t) + \Gamma_2\left((F_2 \bar{x}(t) + \sigma)^{\ast} \Pi(T - t)\right) + B^{\ast} q(T - t)\right],
\end{equation}
at $t\in\mfT$ and agent $i=1$ is allowed to choose an arbitrary control process $u^{[N]}_1 \in \mc{U}^{[N]}$. Here, for clarity we use the superscript $[N]$ to denote the processes associated with the $N$-player game. In this context, the dynamics of agent $i=1$ and agent $i$, $i\in\mc{N}$ and $i\neq 1$, in the $N$-player game are, respectively, given by 
\begin{align} 
x_{1}^{[N]}(t)=&S(t)\xi_1+\int_{0}^{t}S(t-r)\left (Bu^{[N]}_1(r)+F_1x^{(N)}(r)  \right )dr\notag \allowdisplaybreaks\\&+\int_{0}^{t}S(t-r)(Dx_{1}^{[N]}(r)+Eu^{[N]}_1(r)+F_2x^{(N)}(r)+\sigma)dW_{1}(r), \label{fini1} \allowdisplaybreaks\\
x_{i}^{[N]}(t)=&S(t)\xi_i+\int_{0}^{t}S(t-r)\left (Bu^{[N],\circ }_{i}(r)+F_1x^{(N)}(r)  \right )dr\notag \allowdisplaybreaks\\&+\int_{0}^{t}S(t-r)(Dx_{i}^{[N]}(r)+Eu^{[N],\circ}_{i}(r)+F_2x^{(N)}(r)+\sigma)dW_{i}(r) , \label{finii}
\end{align}
where $x^{(N)}(t):=\frac{1}{N} \sum_{i\in \mc{N}}x_{i}^{[N]}(t)$ is the average state of $N$ agents. % and the feedback laws \( \{u^{[N],\circ}_i\}_{ i \in \mathcal{N}} \), given by \eqref{nast}–\eqref{nast-2}, are expressed as
%\begin{equation} 
%u^{[N],\circ}_i(t) = -K^{-1}(T - t)\left[L(T - t)\, x^{[N]}_i(t) + \Gamma_2\left((F_2 \bar{x}(t) + \sigma)^{\ast} \Pi(T - t)\right) + B^{\ast} q(T - t)\right],
%\end{equation}
We note that for any control process $u^{[N]}_1 \in \mc{U}^{[N]}$, the coupled system described by \eqref{fini1}-\eqref{finii} satisfies \Cref{as1}-\Cref{as5}. Thus, the well-posedness of the system is ensured.  The cost functional of agent $i=1$ in the $N$-player game is given by  
\begin{align} 
\label{cost-eps}
J^{[N]}_1(u^{[N]}_{1},u_{-1}^{[N],\circ}):=& \mathbb{E}\int_{0}^{T}\!\!\!\left(\left|M^{\frac{1}{2}}\left (x_{1}^{[N]}(t)-\widehat{F}_1x^{(N)}(t)  \right ) \right|^{2}+\left| u^{[N]}_{1}(t) \right|^{2}\right)dt\notag\\&+\mathbb{E}\left|G^{\frac{1}{2}}\left (x_{1}^{[N]}(T)-\widehat{F}_2x^{(N)}(T)  \right ) \right|^{2}.
\end{align}

We note that the sequence of \( N \)-player games described by \eqref{fini1} and \eqref{finii}, as \( N \) ranges over \( \mathbb{N} \), is associated with the sequence of control processes  \( \big\{ u^{[N]}_1 \big\}_{N \in \mathbb{N}} \) employed by agent \( i = 1 \).

At the equilibrium, where agent $i=1$ employs the strategy $u_1^{[N],\circ}$ given by \eqref{equil-strategy-N-player} for $i=1$, we denote its state by ${x}_1^{[N],\circ}$ and the corresponding average state of $N$ agents by ${x}^{(N),\circ}$. 
The $\epsilon$-Nash property indicates that 
\begin{equation} \label{epnad}
  J^{[N]}_1(u^{[N],\circ}_{1},u^{[N],\circ}_{-1}) \leq  \inf_{ \left\{u^{[N]}_1 \in \mc{U}^{[N]} \right\} }  J^{[N]}_1(u^{[N]}_{1},u^{[N],\circ}_{-1})+\epsilon_N,
\end{equation}
where the sequence $\{\epsilon_N\}_{N\in \mb{N}}$ converges to zero. %We point out that the superscript in \( u^{[N]}_{1} \) is only needed for our analysis; therefore, \eqref{epnad} can be equivalently expressed as
% \begin{equation}
%  J^{[N]}(u^{\circ}_{1},u^{\circ}_{-1}) \leq  \inf_{ \left\{u_1 \in \mc{U}^{[N]} \right\} }  J^{[N]}(u_{1},u^{\circ}_{-1})+\epsilon_N   
% \end{equation}
To establish this property, we start by identifying relevant bounds  for the systems described by \eqref{fini1}-\eqref{finii} for a fixed number of agents $N$ and a given deviating control $u^{[N]}_1 \in \mathcal{U}^{[N]}$ for agent $i = 1$. These bounds are detailed in \Cref{thm:lemmasum}, \Cref{conatm}, and \Cref{cconatm}, where, by an abuse of notation, the constant $C({u^{[N]}_1})$\,\footnote{For the sake of notational simplicity, $u^{[N]}_1$ is omitted from $C({u^{[N]}_1})$ in the proofs of \Cref{thm:lemmasum}, \Cref{conatm}, and \Cref{cconatm}.} may vary from one instance to another. We obtain universal bounds %(for any $N\in \mathbb{N}$ and $u^{[N]}_1 \in \mathcal{U}^{[N]}$) 
for this system in the proof of \Cref{thm:E-Nash}. Furthermore, we obtain the relevant bounds associated with the equilibrium, which are independent of $N$. These bounds are denoted by $C^\circ$, by an abuse of notation.
\begin{lemma}\label{thm:lemmasum}
Consider the $N$ coupled systems described by \eqref{fini1}-\eqref{finii}. Then, the property
\begin{equation}\label{prop-boundedness}
\mathbb{E}\left [\sum_{i\in \mc{N}}\left|x_{i}^{[N]}(t) \right|^2_H  \right ] \leq C({u^{[N]}_1})\,N ,  
\end{equation}
holds uniformly for all \(t \in \mfT\). Here, the constant $C({u^{[N]}_1})$ depends on the model parameters and $u^{[N]}_1$.
\end{lemma}
\begin{proof}
From \eqref{finii}, for agent $i$, $i\in\mc{N}$ and $i\neq 1$, by a simple computation, we have 
\begin{align} \label{boundi}
\mathbb{E}&\left| x_{i}^{[N]}(t)\right|^{2} \\
&\!\leq C\, \mathbb{E}\left[\left | \xi_i \right |^2\!+\!\!\int_{0}^{t}\left(\left|Bu^{[N],\circ}_{i}(r) \!+\!F_1x^{(N)}(r)\right|^{2}\!\!\!+\!\left\|Dx_{i}^{[N]}(r)+Eu^{[N],\circ}_{i}(r) +F_2x^{(N)}(r)+\sigma\right\|_{\mathcal{L}_2}^{2}\right)dr\right] \notag \\  &\leq  C \Big(\int_{0}^{t}\mathbb{E}\left|x_{i}^{[N]}(r)\right|^{2}dr+\int_{0}^{t}\mathbb{E}\left|x^{(N)}(r)\right|^{2}dr+1\Big) \notag \\  &\leq C \Big(\int_{0}^{t}\mathbb{E}\left|x_{i}^{[N]}(r)\right|^{2}dr+\frac{1}{N}\int_{0}^{t}\mathbb{E}\bigg [\sum_{j\in \mc{N}}\left|x_{j}^{[N]}(r) \right|^{2}  \bigg ]dr+1\Big).\notag
\end{align}
From \eqref{fini1}, for agent $i = 1$ we have
\begin{equation} \label{bound1}
\mathbb{E}\left|x_{1}^{[N]}(t) \right|^{2} \leq C \Big( \int_{0}^{t}\mathbb{E}\left|x_{1}^{[N]}(r)\right|^{2}dr+\frac{1}{N}\int_{0}^{t}\mathbb{E}\Big [\sum_{j\in \mc{N}}\left|x_{j}^{[N]}(r) \right|^{2}  \Big]dr+1\Big).    
\end{equation}
From \eqref{boundi} and \eqref{bound1}, we obtain
\begin{equation}
\mathbb{E}\left [\sum_{i\in \mc{N}}\left|x_{i}^{[N]}(t) \right|^{2}\right ] \leq 
C \left (N+\int_{0}^{t}\mathbb{E}\left [\sum_{i\in \mc{N}}\left|x_{i}^{[N]}(r) \right|^{2}  \right ]dr  \right ).
\end{equation}
Applying Grönwall's inequality to the above equation results in \eqref{prop-boundedness}.
\end{proof}
Note that \Cref{thm:lemmasum} is closely related to a part of \Cref{enu} (see \eqref{gammabou}), and it also demonstrates that the solution of the  system \eqref{fini1}-\eqref{finii} belongs to \(\mathcal{H}^2(\mfT;H^N)\)
. A similar argument is presented in \cite[Theorem 9.1]{da2014stochastic}. As direct consequences of \Cref{thm:lemmasum}, we have
\begin{equation}
\mathbb{E}\left|x^{(N)}(t)\right|^{2} \leq C({u^{[N]}_1}),\quad \mathbb{E}\left|x_{1}^{[N]}(t)\right|^{2} \leq C({u^{[N]}_1}),\quad \forall t \in \mfT.
\end{equation}
It is straightforward to verify that, at the equilibrium, we have 
\begin{equation} \label{steq}
\mathbb{E}\left|x^{(N),\circ}(t)\right|^{2} \leq C^{\circ},\quad \mathbb{E}\left|x_{1}^{[N],\circ}(t)\right|^{2} \leq C^{\circ},\quad \forall t \in \mfT,\quad \forall N \in \mathbb{N}.\end{equation}
The next theorem demonstrates the convergence of the average state $x^{(N)}(t)$ to the mean field $\bar{x}(t)$.
\begin{theorem}[Average State Error Bound] \label{conatm}
Suppose the state of any agent $i$, $i\in\mc{N}$ and $i\neq 1$, satisfies \eqref{finii}, where the agent employs the strategy $u^{[N],\circ}_i$ given by \eqref{equil-strategy-N-player}. For any control process $u_1^{[N]} \in \mc{U}^{[N]}$ that agent $i=1$ chooses, we have 
\begin{equation}\label{convergence}
\displaystyle \sup_{t\in \mfT}\mathbb{E}\left|\bar{x}(t)- x^{(N)}(t) \right|_{H}^{2} \leq \frac{C({u^{[N]}_1})}{N}\,.
\end{equation}
\end{theorem}
\begin{proof}
From \cref{Nash-eq}, recall that $\tau (t)\!=\!K^{-1}(T\!-\!t)\!\left [ B^{\ast }q(T-t)\!+\!\Gamma _2((F_2\bar{x}(t)\!+\!\sigma)^{\ast}\Pi (T\!-\!t))\right]$, and 
\begin{equation} 
\bar{x}(t)=S(t)\bar{\xi}- \int_{0}^{t}S(t-r)(BK^{-1}(T-r)L (T-r)\bar{x}(r)+B\tau(r)-F_1\bar{x}(r))dr. 
\end{equation} 
Moreover, from \eqref{finii} and \eqref{fini1} subject to any control process $u_1^{[N]} \in \mc{U}^{[N]}$, we have
\begin{align} 
x^{(N)}(t)=&S(t)x^{(N)}(0)- \int_{0}^{t}S(t-r)(BK^{-1}(T-r)L (T-r)x^{(N)}(r)+B\tau(r)-F_1x^{(N)}(r))dr \notag \allowdisplaybreaks\\&+\frac{1}{N}\bigg [\sum_{i\in \mc{N}}\Xi_{i}(t) \bigg ]+\frac{1}{N}\int_{0}^{t}S(t-r)B(u_1^{[N]}(r)+K^{-1}(T-r)L (T-r)x_{1}^{[N]}(r)+B\tau(r))dr,
\end{align} 
where the stochastic convolution processes $\Xi_{1}(t)$ and $\Xi_{i}(t)$, $i\in\mc{N}$ and $i\neq 1$, are, respectively, given by
\begin{align}
&\Xi_{1}(t)=\int_{0}^{t}S(t-r)(Dx_{1}^{[N]}(r)+Eu_1^{[N]}(r)+F_2x^{(N)}(r)+\sigma)dW_{1}(r),\notag \\ & \Xi_{i}(t)=\int_{0}^{t}S(t-r)\Big[(D-EK^{-1}(T-r)L (T-r))x_{i}^{[N]}(r)-E\tau(r)+F_2x^{(N)}(r)+\sigma \Big]dW_{i}(r). \label{scin1}
\end{align}
Now, define $y(t):=\bar{x}(t)-x^{(N)}(t)$. Then, we have  
\begin{align}
y(t)=  &S(t)y(0)-\int_{0}^{t}S(t-r)(BK^{-1}(T-r)L (T-r)-F_1)y(r)dr
-\frac{1}{N}\bigg [\sum_{i\in \mc{N}}\Xi_{i}(t)\bigg ]\notag \\&-\frac{1}{N}\int_{0}^{t}S(t-r)B(u_1^{[N]}(r)+K^{-1}(T-r)L (T-r)x_{1}^{[N]}(r)+B\tau(r))dr.
\end{align}
Furthermore, from the above equation we obtain
\begin{align} \label{proce}
\mathbb{E}\left|y(t) \right|^{2} \leq & \,\,C\Bigg(\left|y(0) \right|^2+  \int_{0}^{t}\mathbb{E}\left|y(r)\right|^{2}dr+\frac{1}{N^{2}}\bigg[\mathbb{E}\Big|\sum_{i\in \mc{N}}\Xi_{i}(t)   \Big|^{2}\notag \allowdisplaybreaks\\ &+\int_{0}^{t}\mathbb{E}\left|u_1^{[N]}(r)+K^{-1}(T-r)L (T-r)x_{1}^{[N]}(r)+B\tau(r)\right| ^2dr \bigg]\Bigg).
\end{align}

Moreover, since $\Xi_{i}(t),\, i\in \mc{N},$ are driven by independent $Q$-Wiener processes, we have 
\begin{equation} \label{orth}
\mathbb{E}\left|\sum_{i\in \mc{N}}\Xi_{i}(t)   \right|^{2}= \mathbb{E}\left [\sum_{i\in \mc{N}}\left|\Xi_{i}(t)   \right|^{2}\right].    
\end{equation}
More specifically, in the above equation, we use the property that \(\mathbb{E}\left\langle \Xi_{i}(t), \Xi_{j}(t) \right\rangle_H = 0\) for \(i \neq j\) and for all \(i, j \in \mathcal{N}\), and for every \(t \in \mfT\). A straightforward method to verify this property is to demonstrate that it holds for stochastic integrals of elementary processes. This can be achieved by applying the same techniques used to prove the Itô isometry (see, e.g., \cite[Proposition 4.20]{da2014stochastic} and \cite[Proposition 2.1]{gawarecki2010stochastic}).

From \eqref{scin1} and using the standard approximation technique for stochastic convolutions, and given that all operators are uniformly bounded on \(\mfT\), $\forall i\in\mc{N}$ and $i\neq 1$, and $\forall t \in \mfT$, we obtain
\begin{align}
\mathbb{E}\left|\Xi_{i}(t)   \right|^{2} &\leq C\int_{0}^{t}\mathbb{E}\left\|(D-EK^{-1}L (T-r))x_{i}^{[N]}(r)-E\tau(r)+F_2x^{(N)}(r)+\sigma \right\|^{2}dr \notag \allowdisplaybreaks\\&\leq C\int_{0}^{t}(\mathbb{E}\left|x_{i}^{[N]}(r)\right|^{2}+\mathbb{E}\left|x^{(N)}(r)\right|^{2}+1)dr.
\end{align}
Similarly, for $\Xi_{1}(t), \forall t \in \mfT$, we have 
\begin{equation}
\mathbb{E}\left|\Xi_{1}(t)   \right|^{2} \leq C\int_{0}^{t}(\mathbb{E}\left|x_{1}^{[N]}(r)\right|^{2}+\mathbb{E}\left|x^{(N)}(r)\right|^{2}+1)dr.
\end{equation}
Subsequently, we obtain
\begin{equation}
\mathbb{E}\sum_{i\in \mc{N}}\left|\Xi_{i}(t)   \right|^{2} \leq C  \left (\int_{0}^{t}\mathbb{E}\left [\sum_{i\in \mc{N}}\left|x_{i}^{[N]}(r) \right|^2\right ]dr+N \int_{0}^{t} \mathbb{E}\left|x^{(N)}(r)\right|^{2} dr +N   \right ) \leq CN. \label{term3}
\end{equation}
Moreover, for the last term on the RHS of \eqref{proce}, we have 
\begin{align}
\int_{0}^{t}\mathbb{E}\left|u_1^{[N]}(r)+K^{-1}(T-r)L (T-r)x_{1}^{[N]}(r)+B\tau(r) \right| ^2dr \leq C\int_{0}^{t}(\mathbb{E}\left|x_{1}^{[N]}(r)\right|^{2}+1)dr \leq C. \label{term4}
\end{align}
From \eqref{proce} and \eqref{term3}-\eqref{term4}, we conclude that
\begin{equation}
\mathbb{E}\left|y(t) \right|^{2} \leq C(\frac{1}{N}+\frac{1}{N^2})+ C\int_{0}^{t}\mathbb{E}\left|y(r)\right|^{2}dr.
\end{equation}
Then, by Grönwall's inequality, the property \eqref{convergence} follows.
\end{proof}
\begin{proposition}[Error Bounds for Agent $i=1$] \label{cconatm}
Let $x_1(t)$ and $x_1^{[N]}(t)$, respectively, denote the state of agent $i=1$ in the limiting game and the $N$-player game satisfying \eqref{dylimi-S0} and \eqref{fini1}. Moreover, let $J^{\infty }_1(.)$ and $J^{[N]}_1(.,u^{\circ}_{-1})$, respectively, denote the cost functional of agent $i=1$ in the limiting game and the $N$-player game given by \eqref{cosfinlimi-S0} and \eqref{cost-eps}. 
\begin{itemize}
    \item[(i)] If agent $i=1$ employs the control law given by \eqref{nast}, we have
\begin{gather}
\displaystyle \sup_{t\in \mfT}\mathbb{E}\left|x_1^{\circ}(t)-{x}_1^{[N],\circ}(t)\right|_{H}^{2} \leq \frac{C^{\circ}}{N},\label{state-conv-1}\allowdisplaybreaks\\  
\left|J^{\infty }_1(u^{\circ }_{1})-J^{[N]}_1(u^{[N],\circ }_{1},u^{[N],\circ}_{-1}) \right|   \leq \frac{C^{\circ}}{\sqrt{N}}.\label{cost-conv-1} 
\end{gather}
\item[(ii)] If agent $i=1$ employs any control process $u^{[N]}_1 \in \mc{U}^{[N]}$, we have 
\begin{gather}\label{state-conv-2}
\displaystyle \sup_{t\in \mfT}\mathbb{E}\left|x_{1}(t)-x_{1}^{[N]}(t)\right|_{H}^{2} \leq \frac{C({u^{[N]}_1})}{N},\allowdisplaybreaks\\
\left|J^{\infty}_1(u^{[N]}_1)-J^{[N]}_1(u^{[N]}_1,u^{[N],\circ}_{-1}) \right|   \leq \frac{C({u^{[N]}_1})}{\sqrt{N}}.\label{cost-conv-2}
\end{gather}
\end{itemize}
\end{proposition}
% \begin{gather}\label{state-conv-2}
% \left|J^{\infty}_1(u_1^{[N]})-J^{[N]}_1(u^{[N]}_1,u^{[N],\circ}_{-1}) \right|   \leq \frac{C({u^{[N]}_1})}{\sqrt{N}},\allowdisplaybreaks\\
% \left|J^{\infty}_1(u_1^{[N]})-J^{[N]}_1(u^{[N]}_1,u^{\circ}_{-1}) \right|   \leq \frac{C({u^{[N]}_1})}{\sqrt{N}}.\label{cost-conv-2}
% \end{gather}
% \begin{gather}\label{state-conv-2}
% \left|J^{\infty}_1(u_1)-J^{[N]}_1(u_1,u^{\circ}_{-1}) \right|   \leq \frac{C^{u_1}}{\sqrt{N}}.\label{cost-conv-2}
% \end{gather}
\begin{proof}
From \eqref{dylimi-S0}, \eqref{MF-eq} and \eqref{fini1}, for the case where agent $i =1$ employs the control law $u^\circ_1$ given by \eqref{nast}-\eqref{nast-2}, by direct computation, we have 
\begin{align}
x_1^{\circ}(t)-x_{1}^{[N],\circ}(t)&=- \int_{0}^{t}S(t-r)BK^{-1}(T-r)L(T-r)(x_1^{\circ}(r)-x_{1}^{[N],\circ}(r))dr\notag\allowdisplaybreaks\\
&\,+\int_{0}^{t}S(t-r)F_1\left (\bar{x}(r)-{x}^{(N),\circ}(r) \right )dr +\int_{0}^{t}S(t-r)F_2\left (\bar{x}(r)-{x}^{(N),\circ}(r)  \right )dW_1(r)\notag\allowdisplaybreaks\\ &\,+\int_{0}^{t}S(t-r)(D-EK^{-1}(T-r)L(T-r))(x_1^{\circ}(r)-{x}_1^{[N],\circ}(r))dW_1(r).
\end{align}
Moreover, for the case where agent $i=1$ employs an arbitrary control process $u^{[N]}_1 \in \mc{U}^{[N]}$, we have
\begin{align}
x_{1}(t)-x^{[N]}(t)=&\int_{0}^{t}S(t-r)F_1\left (\bar{x}(r)-x^{(N)}(r)  \right )dr+\int_{0}^{t}S(t-r)D(x_{1}(r)-x_{1}^{[N]}(r))dW_1(r)\notag \allowdisplaybreaks\\ &+\int_{0}^{t}S(t-r)F_2\left (\bar{x}(r)-x^{(N)}(r)  \right )dW_1(r).   
\end{align}
The rest of the proof for \eqref{state-conv-1} and \eqref{state-conv-2} follows the method used in \Cref{conatm} by taking square norms, expectation, applying Grönwall's inequality (see \eqref{proce}), and leveraging the results of \Cref{conatm}.

For the property \eqref{cost-conv-1}, a simple computation shows that
\begin{equation}
 \left|J^{\infty }_1(u^{\circ }_{1})-J^{[N]}_1(u^{[N],\circ }_{1},u^{[N],\circ}_{-1}) \right| \leq I_1 + I_2,
\end{equation}
\begin{align}
I_1 &= \mathbb{E}\int_{0}^{T}\left| \left|M^{ \frac{1}{2}}(x^{\circ}_1(t)-\widehat{F}_1\bar{x}(t)) \right|^2 - \left|M^{ \frac{1}{2}}(x_{1}^{[N],\circ}(t)-\widehat{F}_1x^{(N),\circ}(t)) \right|^2\right|dt \notag \allowdisplaybreaks\\
&\quad + \mathbb{E}\left|\left|G^{ \frac{1}{2}}(x^{\circ}_1(T)-\widehat{F}_2\bar{x}(T)) \right|^2 - \left|G^{ \frac{1}{2}}(x_{1}^{[N],\circ}(T)-\widehat{F}_2x^{(N),\circ}(T)) \right|^2\right|,
\end{align}
\begin{equation*}    
I_2 = \mathbb{E}\int_{0}^{T}\left|\left|K^{-1}(T-t)L(T-t)x^{\circ}_1(t)+\tau(t) \right|^2 - \left| K^{-1}(T-t)L(T-t)x_{1}^{[N],\circ}(t)+\tau(t) \right|^2\right|dt,
\end{equation*}
where $\tau (t)=K^{-1}(T-t)\left [ B^{\ast }q(T-t)+\Gamma _2((F_2\bar{x}(t)+\sigma)^{\ast}\Pi (T-t))\right]$. For $t \in [0,T)$, we have
\begin{align}   
&\left| \left|M^{ \frac{1}{2}}(x^{\circ}_1(t)-\widehat{F}_1\bar{x}(t)) \right|^2-\left|M^{ \frac{1}{2}}(x_{1}^{[N],\circ}(t)-\widehat{F}_1x^{(N),\circ}(t)) \right|^2\right|  \notag \allowdisplaybreaks\\
&\leq\left\|M^{ \frac{1}{2}}\left [(x^{\circ}_1(t)-\widehat{F}_1\bar{x}(t))- (x_{1}^{[N],\circ}(t)-\widehat{F}_1x^{(N),\circ}(t)) \right ] \right\|^2\notag \allowdisplaybreaks\\
&\,\,\,\,+2\left|M^{ \frac{1}{2}}(x^{\circ}_1(t)-\widehat{F}_1\bar{x}(t))\right|\left|M^{ \frac{1}{2}}\left [(x^{\circ}_1(t)-\widehat{F}_1\bar{x}(t))- (x_{1}^{[N],\circ}(t)-\widehat{F}_1x^{(N),\circ}(t)) \right ] \right|\notag \allowdisplaybreaks\\ & \leq 2\left|M^{ \frac{1}{2}}(x^{\circ}_1(t)-x_{1}^{[N],\circ}(t)) \right|^2+2\left|M^{ \frac{1}{2}}\widehat{F}_1(\bar{x}(t)-x^{(N),\circ}(t)) \right|^2\notag \allowdisplaybreaks\\ 
&\,\,\,\, + 2\left|M^{ \frac{1}{2}}(x^{\circ}_1(t)-\widehat{F}_1\bar{x}(t))\right|\left (2\left|M^{ \frac{1}{2}}(x^{\circ}_1(t)-x_{1}^{[N],\circ}(t)) \right|^2+2\left|M^{ \frac{1}{2}}\widehat{F}_1(\bar{x}(t)-x^{(N),\circ}(t)) \right|^2  \right )^{\frac{1}{2}}.
\end{align}
We apply the same method to the terminal condition in $I_1$. Then, by using the Cauchy–Schwarz inequality, \Cref{conatm}, and  \Cref{cconatm}, we obtain $I_1 \leq \frac{C}{\sqrt{N}}$.
% \begin{equation}
% I^{N}_1 \leq \frac{C}{\sqrt{N}}.
% \end{equation}
We employ the same method as above for $I_2$ to obtain $I_2 \leq \frac{C}{\sqrt{N}}$. 
% \begin{equation}
% I^{N}_2 \leq \frac{C}{\sqrt{N}}.
% \end{equation}

For the property \eqref{cost-conv-2}, we have
\begin{align}
\left|J^{\infty }_1(u^{[N]}_1)-J^{[N]}_1(u^{[N]}_1,u^{[N],\circ}_{-1}) \right| \leq &\mathbb{E}\int_{0}^{T}\left| \left|M^{ \frac{1}{2}}(x_1(t)-\widehat{F}_1\bar{x}(t)) \right|^2 - \left|M^{ \frac{1}{2}}(x_{1}^{[N]}(t)-\widehat{F}_1x^{(N)}(t)) \right|^2\right|dt \notag \allowdisplaybreaks\\
&+ \mathbb{E}\left|\left|G^{ \frac{1}{2}}(x_1(T)-\widehat{F}_2\bar{x}(T)) \right|^2 - \left|G^{ \frac{1}{2}}(x_{1}^{[N]}(T)-\widehat{F}_2x^{(N)}(T)) \right|^2\right|.
\end{align}
Then, we repeat the same method as for the property \eqref{cost-conv-1} to obtain the property \eqref{cost-conv-2}.
\end{proof}
\begin{remark}
We note that \( \big\{C(u^{[N]}_1)\big\}_{N \in \mathbb{N}}\) is a sequence of real numbers, although $C(u^{[N]}_1)$ does not explicitly depend on \( N \). Therefore, the convergence properties given by \eqref{convergence}, \eqref{state-conv-2}, and \eqref{cost-conv-2} hold as $N\rightarrow \infty$, provided the sequence is bounded. For instance, this condition is met if \( \mathbb{E}\left[\int_{0}^{T} \left|u^{[N]}_{1}(t)\right|^2 \, dt\right] \) is uniformly bounded across all \( N\in \mathbb{N}\), or if the system is at equilibrium.
% We point out that \( C(u^{[N]}_1) \) is a sequence of real numbers, although it does not explicitly depend on \( N \). Therefore, \eqref{convergence}, \eqref{state-conv-2}, and \eqref{cost-conv-2} satisfy the convergence condition if the sequence \( C(u^{[N]}_1) \) is bounded. For instance, this holds if \( \mathbb{E}\left[\int_{0}^{T} \left|u^{[N]}_{1}(t)\right|^2 \, dt\right] \) is uniformly bounded for all \( N \in \mathbb{N}\), or if the system is at equilibrium.
\end{remark}
Now, we establish the $\epsilon$-Nash property. 
\begin{theorem} ($\epsilon$-Nash Equilibrium)\label{thm:E-Nash} Suppose that \Cref{init-cond-iid} and condition \eqref{concondi} hold. Then, the set of control laws $\{u^{[N],\circ}_i\}_{i \in \mc{N}}$, where  $u^{[N],\circ}_i(t)$ is given by 
\eqref{equil-strategy-N-player},
forms an $\epsilon$-Nash equilibrium for the $N$-player system described by \eqref{dyfin}--\eqref{cosfinN}. That is, if agent $i=1$ unilaterally deviates from this set of control laws, then there is a sequence of nonnegative numbers $\{\epsilon_N\}_{N\in \mb{N}}$ converging to zero, such that 
 \begin{equation} \label{epnash1}
  J^{[N]}_1(u^{[N],\circ}_{1},u^{[N],\circ}_{-1}) \leq  \inf_{ \left\{u^{[N]}_1 \in \mc{U}^{[N]} \right\} } J^{[N]}_1(u^{[N]}_1,u^{[N],\circ}_{-1})+\epsilon_N,
\end{equation} 
where $\epsilon_N =O(\tfrac{1}{\sqrt{N}})$.
\end{theorem}
\begin{proof}
It is evident that the $\epsilon$-Nash property \eqref{epnash1} can be equivalently expressed as
\begin{equation} \label{epnash2}
  J^{[N]}_1(u^{[N],\circ}_{1},u^{[N],\circ}_{-1}) \leq  \inf_{ \left\{u^{[N]}_1 \in \mc{A}^{[N]} \right\} }  J^{[N]}_1(u^{[N]}_1,u^{[N],\circ}_{-1})+\epsilon_N,
\end{equation}
where $\mc{A}^{[N]}:=\left\{u^{[N]}_1 \in \mc{U}^{[N]}: J^{[N]}_1(u^{[N]}_1,u^{[N],\circ}_{-1}) \leq  J^{[N]}_1(u^{[N],\circ}_{1},u^{[N],\circ}_{-1})  \right\}$ is a non-empty set.
We then note that, by \eqref{steq}, the equilibrium cost $J^{[N]}_1(u^{[N],\circ}_{1},u^{[N],\circ}_{-1})$ is uniformly bounded by a constant, denoted by $C^{\circ \circ}$, for all $N \in \mathbb{N}$. Thus, for any $N \in \mathbb{N}$ and for any $u^{[N]}_1 \in \mc{A}^{[N]}$, we have
\begin{equation} \label{boundu}
\mathbb{E}\left [\int_{0}^{T}\left |u^{[N]}_1(t) \right |^2dt  \right ]\leq   J^{[N]}_1(u^{[N]}_1,u^{[N],\circ}_{-1}) \leq  J^{[N]}_1(u^{[N],\circ}_{1},u^{[N],\circ}_{-1}) \leq C^{\circ \circ}.
\end{equation}
From \eqref{boundu}, it is straightforward to verify that the constants $C({u^{[N]}_1})$ in \Cref{thm:lemmasum}, \Cref{conatm} and \Cref{cconatm} can be chosen universally for all $N\in \mathbb{N}$ and $u^{[N]}_1 \in \mc{A}^{[N]}$ in each inequality. Therefore, for \eqref{cost-conv-2}, we  have 
\begin{equation} \label{costconv3}
\left|J^{\infty }_1(u^{[N]}_1)-J^{[N]}_1(u^{[N]}_1,u^{[N],\circ}_{-1}) \right|   \leq \frac{C^{\ast}}{\sqrt{N}}, \quad \forall N \in \mathbb{N},\quad \forall u^{[N]}_1 \in \mc{A}^{[N]},
\end{equation}
where $C^{\ast}$ only depends on the model parameters and the constant $C^{\circ \circ}$. Therefore, we have \begin{equation} 
J^{\infty }_1(u^{[N]}_1) \leq J^{[N]}_1(u^{[N]}_1,u^{[N],\circ}_{-1})+\frac{C^{\ast}}{\sqrt{N}}, \quad \forall N \in \mathbb{N},\quad \forall u^{[N]}_1 \in \mc{A}^{[N]},
\end{equation}
and hence 
\begin{equation} \label{ineN}
J^{\infty }_1(u^{[N]}_1) \leq \inf_{ \left\{u^{[N]}_1 \in \mc{A}^{[N]} \right\} } J^{[N]}_1(u^{[N]}_1,u^{[N],\circ}_{-1})+\frac{C^{\ast}}{\sqrt{N}},\quad \forall N \in \mathbb{N},\quad \forall u^{[N]}_1 \in \mc{A}^{[N]}.
\end{equation}
We then recall from \Cref{sincon} and \Cref{fixp} that
\begin{equation}
 J^{\infty }_1(u^{\circ}_1) \leq J^{\infty }_1(u^{[N]}_1),\quad \forall N \in \mathbb{N}, \quad \forall u^{[N]}_1 \in \mc{A}^{[N]}.
\end{equation}
Next, from \eqref{cost-conv-1} and \eqref{ineN}, we have
\begin{equation} 
J^{[N]}_1(u^{[N],\circ }_{1},u^{[N],\circ}_{-1})-\frac{C^{\circ}}{\sqrt{N}} \leq  J^{\infty }_1(u^{\circ}_1) \leq     J^{\infty }_1(u^{[N]}_1) \leq \inf_{ \left\{u^{[N]}_1 \in \mc{A}^{[N]} \right\} } J^{[N]}_1(u^{[N]}_1,u^{[N],\circ}_{-1})+\frac{C^{\ast}}{\sqrt{N}},
\end{equation}
which finally gives 
\begin{equation}
J^{[N]}_1(u^{[N],\circ }_{1},u^{[N],\circ}_{-1}) \leq \inf_{ \left\{u^{[N]}_1 \in \mc{A}^{[N]} \right\} } J^{[N]}_1(u^{[N]}_1,u^{[N],\circ}_{-1})+\frac{C^{\ast}+C^{\circ}}{\sqrt{N}}.
\end{equation}
\end{proof}
\section {Concluding Remarks}\label{conclusion}
We conclude the paper by studying a toy model and introducing a slight generalization of our framework that could broaden its applicability.  
\subsection{A Toy Model}
We now study a toy model
inspired by the model presented in \cite{fouque2018mean}, where the dynamics of a representative agent indexed by $i, i\in \mc{N},$ is given by
\begin{align} \label{dylimitoy}
 dx_i(t)&=(Ax_i(t)+Bu_i(t))dt+\sigma dW_{i}(t) , \notag \\ 
 x_i(0)&=\xi.  
\end{align}
Moreover, agent $i$ aims to minimize the cost functional
\begin{equation}
J^{[N]}(u_i, u_{-i})=\mathbb{E}\int_{0}^{T}\left(\left|M^{\frac{1}{2}}\left (x_{i}(t)-x^{(N)}(t)  \right ) \right|^{2}+\left|u_{i}(t) \right|^{2}\right)dt+\mathbb{E}\left|G^{\frac{1}{2}}\left (x_{i}(T)-x^{(N)}(T)  \right ) \right|^{2}.  
\end{equation}
According to \eqref{dyfin}-\eqref{cosfinN}, we have $F_1 = F_2 = D = E = 0$, and $\widehat{F}_1=\widehat{F}_2=I$ for the above model. 
Applying our results, the contraction condition \eqref{concondi} simplifies to
\begin{equation}
TC_6 \mathrm{exp}(4TM_{T}C_6)<1, \notag 
\end{equation}
where $C_6=2M_{T}^{2} \left\|B \right\|^{2}\left ( \left \| G \right \|+T\left \| M \right \|  \right )$. To find a solution for a fixed $T > 0$, we can adjust the parameters $G$ and $M$ to ensure that the contraction condition is satisfied. Then the $\epsilon$-Nash equilibrium is given by $\{u^{\circ}_i\}_{i \in \mc{N}}$, where
\begin{gather} 
u^{\circ}_i(t) = -B^{\ast}(\Pi(T-t)x_i(t) +q(T-t)),\label{opticons-toy}\\
\bar{x}(t)=S(t)\xi- \int_{0}^{t}S(t-r)B^{\ast}(\Pi(T-t)\bar{x}(t) +q(T-t) )dr, \\ 
\frac{d}{dt}\left<\Pi(t)x, x\right> = 2\left<\Pi(t)x, Ax\right> - \left<\Pi(t)BB^{\ast}\Pi(t)x, x\right> + \left<Mx, x\right>,  \label{ricas-toy}\\
\dot{q}(t) = \left(A^* -\Pi(t)BB^*\right)q(t) - M\bar{x}(T-t), \label{qts-toy}
\end{gather}
with $\Pi(0) = G$, $x \in \mathcal{D}(A)$ and $q(0) = -G\bar{x}(T)$.
\subsection{A Slight Generalization }
Recall that we have defined $D \in \mathcal{L}(H, \mathcal{L}(V, H))$, $E \in \mathcal{L}(U, \mathcal{L}(V, H))$, $F_2 \in \mathcal{L}(H; \mathcal{L}(V; H))$, and $\sigma \in \mathcal{L}(V; H)$ for the volatility in \eqref{dyfin}. As mentioned in \Cref{integrand}, this setting might be more restrictive than necessary. One reason for this conservatism is our desire to align, in particular the derivations of \Cref{sincon}, with the foundational literature, notably references such as \cite{ichikawa1979dynamic} and \cite{ichikawa1982stability}. These settings could potentially be generalized to bounded linear operators from $H$ and $U$ to $\mathcal{L}_2(V_Q, H)$, with $\sigma$ also set as $\sigma \in \mathcal{L}_2(V_Q, H)$. The conclusions of this paper could likely be achieved with only minor modifications.

\bibliographystyle{elsarticle-num-names} 

\bibliography{sample}

\begin{thebibliography}{47}
\expandafter\ifx\csname natexlab\endcsname\relax\def\natexlab#1{#1}\fi
\providecommand{\url}[1]{\texttt{#1}}
\providecommand{\href}[2]{#2}
\providecommand{\path}[1]{#1}
\providecommand{\DOIprefix}{doi:}
\providecommand{\ArXivprefix}{arXiv:}
\providecommand{\URLprefix}{URL: }
\providecommand{\Pubmedprefix}{pmid:}
\providecommand{\doi}[1]{\href{http://dx.doi.org/#1}{\path{#1}}}
\providecommand{\Pubmed}[1]{\href{pmid:#1}{\path{#1}}}
\providecommand{\bibinfo}[2]{#2}
\ifx\xfnm\relax \def\xfnm[#1]{\unskip,\space#1}\fi
%Type = Article
\bibitem[{Huang et~al.(2006)Huang, Malham{\'e}, and Caines}]{huang2006large}
\bibinfo{author}{M.~Huang}, \bibinfo{author}{R.~P. Malham{\'e}}, \bibinfo{author}{P.~E. Caines},
\newblock \bibinfo{title}{Large population stochastic dynamic games: {C}losed-loop {McKean-Vlasov} systems and the {Nash} certainty equivalence principle},
\newblock \bibinfo{journal}{Communications in Information \& Systems} \bibinfo{volume}{6} (\bibinfo{year}{2006}) \bibinfo{pages}{221--252}.
%Type = Article
\bibitem[{Huang et~al.(2007)Huang, Caines, and Malham{\'e}}]{huang2007large}
\bibinfo{author}{M.~Huang}, \bibinfo{author}{P.~E. Caines}, \bibinfo{author}{R.~P. Malham{\'e}},
\newblock \bibinfo{title}{Large-population cost-coupled {LQG} problems with nonuniform agents: individual-mass behavior and decentralized $\varepsilon$-{Nash} equilibria},
\newblock \bibinfo{journal}{IEEE transactions on automatic control} \bibinfo{volume}{52} (\bibinfo{year}{2007}) \bibinfo{pages}{1560--1571}.
%Type = Article
\bibitem[{Lasry and Lions(2007)}]{lasry2007mean}
\bibinfo{author}{J.-M. Lasry}, \bibinfo{author}{P.-L. Lions},
\newblock \bibinfo{title}{Mean field games},
\newblock \bibinfo{journal}{Japanese journal of mathematics} \bibinfo{volume}{2} (\bibinfo{year}{2007}) \bibinfo{pages}{229--260}.
%Type = Book
\bibitem[{Carmona et~al.(2018)Carmona, Delarue et~al.}]{carmona2018probabilistic}
\bibinfo{author}{R.~Carmona}, \bibinfo{author}{F.~Delarue}, et~al., \bibinfo{title}{Probabilistic theory of mean field games with applications I-II}, \bibinfo{publisher}{Springer}, \bibinfo{year}{2018}.
%Type = Book
\bibitem[{Bensoussan et~al.(2013)Bensoussan, Frehse, Yam et~al.}]{bensoussan2013mean}
\bibinfo{author}{A.~Bensoussan}, \bibinfo{author}{J.~Frehse}, \bibinfo{author}{P.~Yam}, et~al., \bibinfo{title}{Mean field games and mean field type control theory}, volume \bibinfo{volume}{101}, \bibinfo{publisher}{Springer}, \bibinfo{year}{2013}.
%Type = Book
\bibitem[{Cardaliaguet et~al.(2019)Cardaliaguet, Delarue, Lasry, and Lions}]{cardaliaguet2019master}
\bibinfo{author}{P.~Cardaliaguet}, \bibinfo{author}{F.~Delarue}, \bibinfo{author}{J.-M. Lasry}, \bibinfo{author}{P.-L. Lions}, \bibinfo{title}{The master equation and the convergence problem in mean field games:({AMS}-201)}, \bibinfo{publisher}{Princeton University Press}, \bibinfo{year}{2019}.
%Type = Article
\bibitem[{Carmona et~al.(2015)Carmona, Fouque, Mousavi, and Sun}]{carmona2013mean}
\bibinfo{author}{R.~Carmona}, \bibinfo{author}{J.-P. Fouque}, \bibinfo{author}{Mousavi}, \bibinfo{author}{L.-H. Sun},
\newblock \bibinfo{title}{Mean field games and systemic risk},
\newblock \bibinfo{journal}{Communications in Mathematical Sciences} \bibinfo{volume}{13} (\bibinfo{year}{2015}) \bibinfo{pages}{911--933}.
%Type = Incollection
\bibitem[{Firoozi and Caines(2017)}]{FirooziISDG2017}
\bibinfo{author}{D.~Firoozi}, \bibinfo{author}{P.~E. Caines},
\newblock \bibinfo{title}{The execution problem in finance with major and minor traders: A mean field game formulation},
\newblock in: \bibinfo{booktitle}{Annals of the International Society of Dynamic Games (ISDG): Advances in Dynamic and Mean Field Games}, volume~\bibinfo{volume}{15}, \bibinfo{publisher}{Birkh\"auser Basel}, \bibinfo{year}{2017}, pp. \bibinfo{pages}{107--130}.
%Type = Article
\bibitem[{Shrivats et~al.(2022)Shrivats, Firoozi, and Jaimungal}]{Firoozi2022MAFI}
\bibinfo{author}{A.~V. Shrivats}, \bibinfo{author}{D.~Firoozi}, \bibinfo{author}{S.~Jaimungal},
\newblock \bibinfo{title}{A mean-field game approach to equilibrium pricing in solar renewable energy certificate markets},
\newblock \bibinfo{journal}{Mathematical Finance} \bibinfo{volume}{32} (\bibinfo{year}{2022}) \bibinfo{pages}{779--824}.
%Type = Article
\bibitem[{Gomes and Saúde(2021)}]{gomes_mean-field_2018}
\bibinfo{author}{D.~Gomes}, \bibinfo{author}{J.~Saúde},
\newblock \bibinfo{title}{A mean-field game approach to price formation in electricity markets},
\newblock \bibinfo{journal}{Dynamic Games and Applications} \bibinfo{volume}{11} (\bibinfo{year}{2021}) \bibinfo{pages}{29--53}.
%Type = Article
\bibitem[{Chang et~al.(2025)Chang, Firoozi, and Benatia}]{chang2022Systemic}
\bibinfo{author}{Y.~Chang}, \bibinfo{author}{D.~Firoozi}, \bibinfo{author}{D.~Benatia},
\newblock \bibinfo{title}{Large banks and systemic risk: Insights from a mean-field game model},
\newblock \bibinfo{journal}{Journal of Systems Science \& Complexity} \bibinfo{volume}{38} (\bibinfo{year}{2025}) \bibinfo{pages}{460–494}.
%Type = Article
\bibitem[{Casgrain and Jaimungal(2020)}]{casgrain_meanfield_2020}
\bibinfo{author}{P.~Casgrain}, \bibinfo{author}{S.~Jaimungal},
\newblock \bibinfo{title}{Mean‐field games with differing beliefs for algorithmic trading},
\newblock \bibinfo{journal}{Mathematical Finance} \bibinfo{volume}{30} (\bibinfo{year}{2020}) \bibinfo{pages}{995--1034}.
%Type = Article
\bibitem[{Carmona et~al.(2022)Carmona, Dayan{\i}kl{\i}, and Lauri{\`e}re}]{carmona2022carbon}
\bibinfo{author}{R.~Carmona}, \bibinfo{author}{G.~Dayan{\i}kl{\i}}, \bibinfo{author}{M.~Lauri{\`e}re},
\newblock \bibinfo{title}{Mean field models to regulate carbon emissions in electricity production},
\newblock \bibinfo{journal}{Dynamic Games and Applications} \bibinfo{volume}{12} (\bibinfo{year}{2022}) \bibinfo{pages}{897--928}.
%Type = Article
\bibitem[{Fouque and Zhang(2018)}]{fouque2018mean}
\bibinfo{author}{J.-P. Fouque}, \bibinfo{author}{Z.~Zhang},
\newblock \bibinfo{title}{Mean field game with delay: a toy model},
\newblock \bibinfo{journal}{Risks} \bibinfo{volume}{6} (\bibinfo{year}{2018}) \bibinfo{pages}{90}.
%Type = Article
\bibitem[{Carmona et~al.(2018)Carmona, Fouque, Mousavi, and Sun}]{carmona2018systemic}
\bibinfo{author}{R.~Carmona}, \bibinfo{author}{J.-P. Fouque}, \bibinfo{author}{S.~M. Mousavi}, \bibinfo{author}{L.-H. Sun},
\newblock \bibinfo{title}{Systemic risk and stochastic games with delay},
\newblock \bibinfo{journal}{Journal of Optimization Theory and Applications} \bibinfo{volume}{179} (\bibinfo{year}{2018}) \bibinfo{pages}{366--399}.
%Type = Book
\bibitem[{Da~Prato and Zabczyk(2014)}]{da2014stochastic}
\bibinfo{author}{G.~Da~Prato}, \bibinfo{author}{J.~Zabczyk}, \bibinfo{title}{Stochastic equations in infinite dimensions}, \bibinfo{publisher}{Cambridge university press}, \bibinfo{year}{2014}.
%Type = Book
\bibitem[{Gawarecki and Mandrekar(2010)}]{gawarecki2010stochastic}
\bibinfo{author}{L.~Gawarecki}, \bibinfo{author}{V.~Mandrekar}, \bibinfo{title}{Stochastic differential equations in infinite dimensions: with applications to stochastic partial differential equations}, \bibinfo{publisher}{Springer Science \& Business Media}, \bibinfo{year}{2010}.
%Type = Article
\bibitem[{Ichikawa(1979)}]{ichikawa1979dynamic}
\bibinfo{author}{A.~Ichikawa},
\newblock \bibinfo{title}{Dynamic programming approach to stochastic evolution equations},
\newblock \bibinfo{journal}{SIAM Journal on Control and Optimization} \bibinfo{volume}{17} (\bibinfo{year}{1979}) \bibinfo{pages}{152--174}.
%Type = Article
\bibitem[{Hu and Tang(2022)}]{hu2022stochastic}
\bibinfo{author}{Y.~Hu}, \bibinfo{author}{S.~Tang},
\newblock \bibinfo{title}{Stochastic {LQ} control and associated {R}iccati equation of {PDE}s driven by state-and control-dependent white noise},
\newblock \bibinfo{journal}{SIAM Journal on Control and Optimization} \bibinfo{volume}{60} (\bibinfo{year}{2022}) \bibinfo{pages}{435--457}.
%Type = Article
\bibitem[{Tessitore(1992)}]{tessitore1992some}
\bibinfo{author}{G.~Tessitore},
\newblock \bibinfo{title}{Some remarks on the {R}iccati equation arising in an optimal control problem with state-and control-dependent noise},
\newblock \bibinfo{journal}{SIAM Journal on Control and Optimization} \bibinfo{volume}{30} (\bibinfo{year}{1992}) \bibinfo{pages}{717--744}.
%Type = Article
\bibitem[{Nurbekyan(2012)}]{nurbekyan2012calculus}
\bibinfo{author}{L.~Nurbekyan},
\newblock \bibinfo{title}{Calculus of variations in hilbert spaces},
\newblock \bibinfo{journal}{Journal of Contemporary Mathematical Analysis} \bibinfo{volume}{47} (\bibinfo{year}{2012}) \bibinfo{pages}{148--160}.
%Type = Article
\bibitem[{Gomes and Nurbekyan(2015)}]{gomes2015minimizers}
\bibinfo{author}{D.~Gomes}, \bibinfo{author}{L.~Nurbekyan},
\newblock \bibinfo{title}{On the minimizers of calculus of variations problems in hilbert spaces},
\newblock \bibinfo{journal}{Calculus of Variations and Partial Differential Equations} \bibinfo{volume}{52} (\bibinfo{year}{2015}) \bibinfo{pages}{65--93}.
%Type = Inproceedings
\bibitem[{Dunyak and Caines(2022)}]{dunyak2022linear}
\bibinfo{author}{A.~Dunyak}, \bibinfo{author}{P.~E. Caines},
\newblock \bibinfo{title}{Linear stochastic graphon systems with {Q}-space noise},
\newblock in: \bibinfo{booktitle}{2022 IEEE 61st Conference on Decision and Control (CDC)}, \bibinfo{organization}{IEEE}, \bibinfo{year}{2022}, pp. \bibinfo{pages}{3926--3932}.
%Type = Article
\bibitem[{Dunyak and Caines(2024)}]{dunyak2024quadratic}
\bibinfo{author}{A.~Dunyak}, \bibinfo{author}{P.~E. Caines},
\newblock \bibinfo{title}{Quadratic optimal control of graphon q-noise linear systems},
\newblock \bibinfo{journal}{arXiv preprint arXiv:2407.00212}  (\bibinfo{year}{2024}).
%Type = Article
\bibitem[{Cosso et~al.(2023)Cosso, Gozzi, Kharroubi, Pham, and Rosestolato}]{cosso2023optimal}
\bibinfo{author}{A.~Cosso}, \bibinfo{author}{F.~Gozzi}, \bibinfo{author}{I.~Kharroubi}, \bibinfo{author}{H.~Pham}, \bibinfo{author}{M.~Rosestolato},
\newblock \bibinfo{title}{Optimal control of path-dependent {McKean--Vlasov SDEs} in infinite-dimension},
\newblock \bibinfo{journal}{The Annals of Applied Probability} \bibinfo{volume}{33} (\bibinfo{year}{2023}) \bibinfo{pages}{2863--2918}.
%Type = Article
\bibitem[{Federico et~al.(2024)Federico, Gozzi, and Ghilli}]{federico2024linear}
\bibinfo{author}{S.~Federico}, \bibinfo{author}{F.~Gozzi}, \bibinfo{author}{D.~Ghilli},
\newblock \bibinfo{title}{Linear-quadratic mean field games in {Hilbert} spaces},
\newblock \bibinfo{journal}{arXiv preprint arXiv:2402.14935}  (\bibinfo{year}{2024}).
%Type = Article
\bibitem[{Bensoussan et~al.(2016)Bensoussan, Sung, Yam, and Yung}]{bensoussan2016linear}
\bibinfo{author}{A.~Bensoussan}, \bibinfo{author}{K.~Sung}, \bibinfo{author}{S.~C.~P. Yam}, \bibinfo{author}{S.-P. Yung},
\newblock \bibinfo{title}{Linear-quadratic mean field games},
\newblock \bibinfo{journal}{Journal of Optimization Theory and Applications} \bibinfo{volume}{169} (\bibinfo{year}{2016}) \bibinfo{pages}{496--529}.
%Type = Article
\bibitem[{Huang(2010)}]{huang2010large}
\bibinfo{author}{M.~Huang},
\newblock \bibinfo{title}{Large-population {{LQG}} games involving a major player: the {Nash} certainty equivalence principle},
\newblock \bibinfo{journal}{SIAM Journal on Control and Optimization} \bibinfo{volume}{48} (\bibinfo{year}{2010}) \bibinfo{pages}{3318--3353}.
%Type = Article
\bibitem[{Firoozi et~al.(2020)Firoozi, Jaimungal, and Caines}]{firoozi2020convex}
\bibinfo{author}{D.~Firoozi}, \bibinfo{author}{S.~Jaimungal}, \bibinfo{author}{P.~E. Caines},
\newblock \bibinfo{title}{Convex analysis for {{LQG}} systems with applications to major--minor {{LQG}} mean--field game systems},
\newblock \bibinfo{journal}{Systems \& Control Letters} \bibinfo{volume}{142} (\bibinfo{year}{2020}) \bibinfo{pages}{104734}.
%Type = Article
\bibitem[{Liu et~al.(2025)Liu, Firoozi, and Breton}]{liu2023lqg}
\bibinfo{author}{H.~Liu}, \bibinfo{author}{D.~Firoozi}, \bibinfo{author}{M.~Breton},
\newblock \bibinfo{title}{{LQG} risk-sensitive single-agent and major-minor mean-field game systems: A variational framework},
\newblock \bibinfo{journal}{SIAM Journal on Control and Optimization} \bibinfo{volume}{63} (\bibinfo{year}{2025}) \bibinfo{pages}{2251--2281}.
%Type = Article
\bibitem[{Firoozi and Caines(2020)}]{firoozi2020epsilon}
\bibinfo{author}{D.~Firoozi}, \bibinfo{author}{P.~E. Caines},
\newblock \bibinfo{title}{$\epsilon$-{Nash} equilibria for major--minor {LQG} mean field games with partial observations of all agents},
\newblock \bibinfo{journal}{IEEE Transactions on Automatic Control} \bibinfo{volume}{66} (\bibinfo{year}{2020}) \bibinfo{pages}{2778--2786}.
%Type = Article
\bibitem[{Firoozi(2022)}]{firoozi2022LQG}
\bibinfo{author}{D.~Firoozi},
\newblock \bibinfo{title}{{{LQG}} mean field games with a major agent: {Nash} certainty equivalence versus probabilistic approach},
\newblock \bibinfo{journal}{Automatica} \bibinfo{volume}{146} (\bibinfo{year}{2022}) \bibinfo{pages}{110559}.
%Type = Article
\bibitem[{Huang(2021)}]{huang2020linear}
\bibinfo{author}{M.~Huang},
\newblock \bibinfo{title}{Linear-quadratic mean field games with a major player: {Nash} certainty equivalence versus master equations},
\newblock \bibinfo{journal}{Communications in Information and Systems, vol. 21, no. 3}  (\bibinfo{year}{2021}).
%Type = Article
\bibitem[{Toumi et~al.(2024)Toumi, Malhamé, and {Le Ny}}]{Malhame_Toumi_2024}
\bibinfo{author}{N.~Toumi}, \bibinfo{author}{R.~Malhamé}, \bibinfo{author}{J.~{Le Ny}},
\newblock \bibinfo{title}{A mean field game approach for a class of linear quadratic discrete choice problems with congestion avoidance},
\newblock \bibinfo{journal}{Automatica} \bibinfo{volume}{160} (\bibinfo{year}{2024}) \bibinfo{pages}{111420}. \URLprefix \url{https://www.sciencedirect.com/science/article/pii/S0005109823005873}. \DOIprefix\doi{https://doi.org/10.1016/j.automatica.2023.111420}.
%Type = Article
\bibitem[{Li et~al.(2023)Li, Mou, Wu, and Zhou}]{li2023linear}
\bibinfo{author}{M.~Li}, \bibinfo{author}{C.~Mou}, \bibinfo{author}{Z.~Wu}, \bibinfo{author}{C.~Zhou},
\newblock \bibinfo{title}{Linear-quadratic mean field games of controls with non-monotone data},
\newblock \bibinfo{journal}{Transactions of the American Mathematical Society} \bibinfo{volume}{376} (\bibinfo{year}{2023}) \bibinfo{pages}{4105--4143}.
%Type = Article
\bibitem[{Firoozi et~al.(2022)Firoozi, Pakniyat, and Caines}]{FIROOZI2022-hybrid}
\bibinfo{author}{D.~Firoozi}, \bibinfo{author}{A.~Pakniyat}, \bibinfo{author}{P.~E. Caines},
\newblock \bibinfo{title}{A class of hybrid {LQG} mean field games with state-invariant switching and stopping strategies},
\newblock \bibinfo{journal}{Automatica} \bibinfo{volume}{141} (\bibinfo{year}{2022}) \bibinfo{pages}{110244}.
%Type = Article
\bibitem[{Firoozi and Jaimungal(2022)}]{FIROOZI2022-exploratory}
\bibinfo{author}{D.~Firoozi}, \bibinfo{author}{S.~Jaimungal},
\newblock \bibinfo{title}{Exploratory {LQG} mean field games with entropy regularization},
\newblock \bibinfo{journal}{Automatica} \bibinfo{volume}{139} (\bibinfo{year}{2022}) \bibinfo{pages}{110177}.
%Type = Book
\bibitem[{Peszat and Zabczyk(2007)}]{peszat2007stochastic}
\bibinfo{author}{S.~Peszat}, \bibinfo{author}{J.~Zabczyk}, \bibinfo{title}{Stochastic partial differential equations with L{\'e}vy noise: An evolution equation approach}, volume \bibinfo{volume}{113}, \bibinfo{publisher}{Cambridge University Press}, \bibinfo{year}{2007}.
%Type = Article
\bibitem[{Fabbri et~al.(2017)Fabbri, Gozzi, and Swiech}]{fabbri2017stochastic}
\bibinfo{author}{G.~Fabbri}, \bibinfo{author}{F.~Gozzi}, \bibinfo{author}{A.~Swiech},
\newblock \bibinfo{title}{Stochastic optimal control in infinite dimension},
\newblock \bibinfo{journal}{Probability and Stochastic Modelling. Springer}  (\bibinfo{year}{2017}).
%Type = Book
\bibitem[{Hyt{\"o}nen et~al.(2016)Hyt{\"o}nen, Van~Neerven, Veraar, and Weis}]{hytonen2016analysis}
\bibinfo{author}{T.~Hyt{\"o}nen}, \bibinfo{author}{J.~Van~Neerven}, \bibinfo{author}{M.~Veraar}, \bibinfo{author}{L.~Weis}, \bibinfo{title}{Analysis in Banach spaces}, volume~\bibinfo{volume}{12}, \bibinfo{publisher}{Springer}, \bibinfo{year}{2016}.
%Type = Book
\bibitem[{Goldstein(2017)}]{goldstein2017semigroups}
\bibinfo{author}{J.~A. Goldstein}, \bibinfo{title}{Semigroups of linear operators and applications}, \bibinfo{publisher}{Courier Dover Publications}, \bibinfo{year}{2017}.
%Type = Article
\bibitem[{Curtain and Falb(1970)}]{curtain1970ito}
\bibinfo{author}{R.~F. Curtain}, \bibinfo{author}{P.~L. Falb},
\newblock \bibinfo{title}{Itô's lemma in infinite dimensions},
\newblock \bibinfo{journal}{Journal of mathematical analysis and applications} \bibinfo{volume}{31} (\bibinfo{year}{1970}) \bibinfo{pages}{434--448}.
%Type = Article
\bibitem[{Ichikawa(1982)}]{ichikawa1982stability}
\bibinfo{author}{A.~Ichikawa},
\newblock \bibinfo{title}{Stability of semilinear stochastic evolution equations},
\newblock \bibinfo{journal}{Journal of Mathematical Analysis and Applications} \bibinfo{volume}{90} (\bibinfo{year}{1982}) \bibinfo{pages}{12--44}.
%Type = Book
\bibitem[{Cohn(2013)}]{cohn2013measure}
\bibinfo{author}{D.~L. Cohn}, \bibinfo{title}{Measure theory}, volume~\bibinfo{volume}{5}, \bibinfo{publisher}{Springer}, \bibinfo{year}{2013}.
%Type = Article
\bibitem[{Da~Prato(1984)}]{da1984direct}
\bibinfo{author}{G.~Da~Prato},
\newblock \bibinfo{title}{Direct solution of a {R}iccati equation arising in stochastic control theory},
\newblock \bibinfo{journal}{Applied Mathematics and Optimization} \bibinfo{volume}{11} (\bibinfo{year}{1984}) \bibinfo{pages}{191--208}.
%Type = Book
\bibitem[{Diagana(2018)}]{diagana2018semilinear}
\bibinfo{author}{T.~Diagana}, \bibinfo{title}{Semilinear evolution equations and their applications}, \bibinfo{publisher}{Springer}, \bibinfo{year}{2018}.
%Type = Book
\bibitem[{Yong and Zhou(1999)}]{yong1999stochastic}
\bibinfo{author}{J.~Yong}, \bibinfo{author}{X.~Y. Zhou}, \bibinfo{title}{Stochastic controls: Hamiltonian systems and HJB equations}, volume~\bibinfo{volume}{43}, \bibinfo{publisher}{Springer Science \& Business Media}, \bibinfo{year}{1999}.

\end{thebibliography}

\end{document}